\documentclass{article}

\usepackage{amssymb,amsmath,amsthm}
\usepackage{texdraw}
\usepackage[latin1]{inputenc}
\usepackage{graphicx}
\usepackage[toc,page]{appendix}
\usepackage{comment}

\numberwithin{equation}{section}
\includecomment{mac}

\def \dt {{\,\mathrm dt}}
\def \dr {{\,\mathrm dr}}
\def \d {{\,\mathrm d}}
\def \dtau {{\,\mathrm d\tau}}
\def \dtheta {{\,\mathrm d\vartheta}}

\def \RR {{\mathbb R}}
\def \NN {{\mathbb N}}
\def \ZZ {{\mathbb Z}}

\newcommand{\GRe}[4]{\Gamma_{#1}}%(#2,#3,#4)}
\newcommand{\GReU}[3]{\Gamma}%(#1,#2,#3)}

\def \ve {\varepsilon}
\def \ep {\varepsilon}
\def \eps {\varepsilon}
\def \vt {\vartheta}
\def \vp {\varphi}

\def \action{\mathcal{A}}
\def \lagr{\mathcal{L}}
\def \ell{\mathcal{E}}
\def \cont{\mathcal{C}}
\def \morse{\mathcal{M}}
\def \pjump {\Delta_{\mathrm{pos}}}
\def \vjump {\Delta_{\mathrm{vel}}}
\newcommand{\sphere}{\mathbb{S}^{d-1}}
\newcommand{\cerchio}{\mathbb{S}^{1}}

\newcommand{\Poh}{\mathcal{P}}
\newcommand{\Seh}{\mathcal{S}}
\newcommand{\setin}{\mathrm{In}}
\newcommand{\setout}{\mathrm{Out}}

\newcommand{\suc}[2]{\left( {#1}_{#2}\right)_{#2}}
\newtheorem{theorem}{Theorem}[section]
\newtheorem{lemma}[theorem]{Lemma}

\newtheorem{definition}[theorem]{Definition}
\newtheorem{proposition}[theorem]{Proposition}
\newtheorem{remark}[theorem]{Remark}
\newtheorem{corollary}[theorem]{Corollary}

\title{\sc Entire Parabolic Trajectories\\
           as Minimal Phase Transitions}
\author{%
Vivina Barutello\footnote{Dipartimento di Matematica, Universit\`a degli Studi
di Torino, Via Carlo Alberto, 10,  10123 Torino,
Italy. e-mail: \texttt{vivina.barutello@unito.it}}
\and
Susanna Terracini\footnote{Dipartimento di Matematica e Applicazioni, Universit\`a degli Studi
di Milano-Bicocca, Via Bicocca degli Arcimboldi, 8, 20126 Milano,
Italy. e-mail: \texttt{susanna.terracini@unimib.it}}
\and
Gianmaria Verzini\footnote{Dipartimento di Matematica, Politecnico di Milano, Piazza
Leonardo da Vinci, 32,  20133 Milano, Italy.
e-mail: \texttt{gianmaria.verzini@polimi.it}}
}
\begin{document}
%%%%%%%%%%%%%%%%%%%%%%%%%%%%%%%%%%%%%%%%%%%%%%%%%%%%%%%%%%%%%%%%%%%%%%%%%%%%%%%
\maketitle
%%%%%%%%%%%%%%%%%%%%%%%%%%%%%%%%%%%%%%%%%%%%%%%%%%%%%%%%%%%%%%%%%%%%%%%%%%%%%%%

\begin{abstract}
For the class of anisotropic Kepler problems in $\RR^d\setminus\{0\}$ with homogeneous potentials, we seek parabolic trajectories having prescribed asymptotic directions at infinity and which, in addition, are Morse minimizing geodesics for the Jacobi metric. Such trajectories correspond to saddle heteroclinics on the collision manifold, are  structurally unstable and appear only for a codimension-one submanifold of such potentials. We give them a variational characterization in terms of the behavior of the parameter-free minimizers of an associated obstacle problem. We then give a full characterization of such a codimension-one manifold of potentials and we show how to parameterize it with respect to the degree of homogeneity.
\end{abstract}

%======================================
\section{Introduction}\label{sec:intro}
%======================================

Let us consider the conservative dynamical system
\begin{equation}\label{eq:dynsys}
\ddot x (t) =\nabla V (x(t)),\qquad x \in \RR^d \setminus \mathcal{X},
\end{equation}
where $d\geq 2$, the potential $V$ is smooth outside --and goes to infinity near-- the collision set $\mathcal{X}$,
and it satisfies the normalization condition
\[
0 = \liminf_{|x| \to \infty} V(x) < V(x) \quad \text{for every } x .
\]
A (global) \emph{parabolic trajectory} for \eqref{eq:dynsys} is a collisionless solution
which has null energy:
\begin{equation}\label{eq:dynsys2}
\frac12 |\dot x(t)|^2 = V(x(t)), \qquad \text{for every } t \in \RR.
\end{equation}
In the Kepler problem ($V(x)=1/|x|$) all global zero-energy
trajectories are indeed parabola.
In celestial mechanics, and more in general in the theory of singular hamiltonian systems,
parabolic trajectories play a central role
and they are known to carry precious information on the behavior of general solutions near collisions
\cite{Chazy2,Chazy1,Chenciner1998,MarSaari76,Pollard67,SaaHulk81,Saari71,Saari84}.
Our aim in this paper is to introduce a
new variational approach to the existence and characterization of such trajectories for homogeneous potentials.

Let us then assume that
\begin{center}
$V$ is homogeneous of degree $-\alpha$, for some $\alpha \in (0,2)$.
\end{center}
In this setting parabolic trajectories can be equivalently defined as solutions
satisfying $|\dot x(t)| \to 0$ as $t \to \pm \infty$ \cite{SaaHulk81,Chenciner1998}, see also Appendix
\ref{app:stime}. Furthermore, such orbits enjoy asymptotic properties, regarding both
$|x|$ and $x/|x|$. First of all $|x(t)|\to\infty$ as $t \to \pm \infty$;
on the other hand, recalling that
a \emph{central configuration} for $V$ is a unitary vector which is a critical
point of the restriction of $V$ to the sphere $\sphere$,
the normalized configuration $x(t)/|x(t)|$
has infinitesimal distance from the set of central configurations of $V$,
as $t \to \pm \infty$. In particular, whenever this set is discrete,
we have that
\[
\frac{x(t)}{|x(t)|} \to \xi^{\pm}, \qquad \text{as } t \to \pm\infty,
\]
where $\xi^{\pm}$ are central configurations.

From this point of view, as enlightened by McGehee in \cite{McG1974},
parabolic trajectories can be seen as heteroclinic connections in
the collision manifold between two asymptotic configurations at infinity (in time and space).
This characterization has been exploited, starting from McGehee, and up to the work of Moeckel \cite{Moeckel89},
in order to study the motion in the three-body problem near collisions.
In the same perspective, an exhaustive study of the planar anisotropic Kepler problem was performed by Devaney. The
potential he considers in \cite{DevInvMath1978,DevProgMath1981} has two pairs of non degenerate
central configurations, corresponding to two minima and two maxima for the restricted potential.
In this situation, parabolic trajectories can be classified into different
types, depending on the limiting directions: the typical ones, which always exist,
connect two central configurations which correspond to maxima;
connections minimum-maximum generically exist and are quite stable objects; finally,
connections minimum-minimum generically do not exist (in the setting of \cite{DevInvMath1978},
these are saddle-saddle heteroclinic connections in the phase plane of the angular variable).
Our aim is to provide conditions for the existence
of this latter kind of trajectories, in terms of a variational characterization
involving a minimization problem.

An interesting interpretation of the existence of  parabolic trajectories which are free Morse minimizers for the action, in the special case of two dimensions, can be given in terms of the weak KAM theory (see \cite{Fathibook,FathiMadernaM07,FathiSinconolfi04}). Indeed, the existence of one minimal entire, collision free, parabolic trajectory induces a lamination of the plane by minimal trajectories (all its rescaled orbits), all homoclinic to the minimal Aubry set (in the present case, the
infinity)  and, correspondingly, leads to the existence of an entire  solution of the associated Hamilton-Jacobi equation on the punctured plane
(see Remark \ref{rem:hj} and also Remark \ref{rem:weakkam} below).

To start with, given any $V$ as above, $a<b$, and $x$ belonging to
the Sobolev space $H^1\left((a,b);\RR^d  \right)$, let us consider the
(possibly infinite) \emph{lagrangian action functional} with lagrangian $\lagr$:
\[
\action(x) = \action([a,b];x):= \int_{a}^{b} \lagr(\dot x(t),x(t)) \, \dt, \qquad
\lagr(\dot x,x):= \frac12 |\dot x|^2 + V(x)
\]
(of course, the action may be finite even though the path $x$ interacts with the
singularity of the potential).
Given $\xi^-$ and $\xi^+$ ingoing and outgoing asymptotic directions, we consider the following class of minimizers.
\begin{definition}\label{defi:Morse_min}
We say that $x\in H^1_{\mathrm{loc}}(\RR)$ is a \emph{(free) minimizer of $\action$ of parabolic type, in the sense of Morse}, if
\begin{itemize}
 \item $\min_{t\in\RR}|x(t)|>0$;
 \item $|x(t)|\to+\infty$, $x(t)/|x(t)|\to\xi^\pm$ as $t\to\pm\infty$;
 \item for every $a<b$, $a'<b'$, and $z\in H^1(a',b')$, there holds
\[
z(a')=x(a),\ z(b')=x(b) \qquad\implies\qquad \action([a,b];x)\leq\action([a',b'];z).
\]
\end{itemize}
\end{definition}
In some situations one may be also interested in Morse minimizers in a local sense,
for instance imposing some topological
constraint. In any case, a parabolic Morse minimizer is of class $\cont^2$ and, because of
Maupertuis' principle, it satisfies the Euler-Lagrange equation \eqref{eq:dynsys}
and the zero-energy relation \eqref{eq:dynsys2}.

We stress the fact that, in general, a potential $V$ does not need to admit a parabolic Morse minimizer.
To deal with this intrinsic structural instability we need to introduce an auxiliary parameter
and look for parabolic orbits as pairs trajectory-parameter. To clarify the role of the additional parameter, it may
be helpful to let the potential vary in a class. As a toy model, we will work on a class shaped on a
multidimensional version of the case described by Devaney, choosing as parameter the homogeneity exponent $-\alpha$.

More precisely, let us fix $\xi^+\neq\xi^-$ in $\sphere$ and $V_{\min}>0$, and let us define the metric spaces
\[
\begin{array}{cl}
\Seh &=\left\{ V\in \cont^2(\sphere):
 \begin{array}{l}
   s\in \sphere \text{ implies }V(s)\geq V(\xi^\pm)=V_{\min};
   \smallskip\\
   \exists \delta>0, \mu>0 \text{ such that }|s-\xi^{\pm}|<\delta
   \smallskip\\
    \text{implies } V(s) - V(\xi^{\pm}) \geq \mu|s-\xi^{\pm}|^2
 \end{array}
\right\},\medskip\\
\Poh &=\left\{ (V,\alpha)\in \cont^2(\sphere)
\times (0,2):\, V\in\Seh\right\},
\end{array}
\]
the latter being equipped with the product distance. With some abuse of notation, we will systematically identify any element of $\Poh$ with the homogeneous extension of its first component:
\[
\Poh\ni(V,\alpha)\quad\leftrightarrow\quad V\in \cont^2(\RR^d\setminus\{0\};\RR),\,
V(x):=\dfrac{V\left(x/|x|\right)}{|x|^\alpha},
\]
in such a way that $\xi^{\pm}$ are non-degenerate, globally minimal central configurations for $V$, which
singular set $\mathcal{X}$ coincides with the origin.

As we will show, the property of a potential to admit parabolic minimizers
is related to its behavior with respect to the following fixed-endpoints problem.
For any $V\in\Poh$, let us define
\[
c(V) := \inf \left\{ \action\left([a,b];x\right) : a<b, \, x \in H^1(a,b), \, x(a)=\xi^-,\, x(b)=\xi^+  \right\};
\]
it is not difficult to prove that such infimum is indeed a minimum,
achieved by a possibly colliding solution. More precisely,
recalling that a homothetic motion is a trajectory with constant angular part,
the following result holds.
\begin{proposition}\label{propo:alter}
Let $V \in \Poh$; then one of the following alternatives is satisfied (see Figure \ref{fig:intro}):
\begin{enumerate}
	\item[(1)] $c(V) = 4\sqrt{2V_{\min}}/(2-\alpha)$ is achieved by the
	juxtaposition of two homothetic motions, the first connecting $\xi^-$ to
	the origin and the second the origin to $\xi^+$;
	\item[(2)] $c(V) < 4\sqrt{2V_{\min}}/(2-\alpha)$,
	and it is achieved by trajectories which are uniformly bounded away from the origin.
\end{enumerate}
\end{proposition}
\begin{figure}[!t]
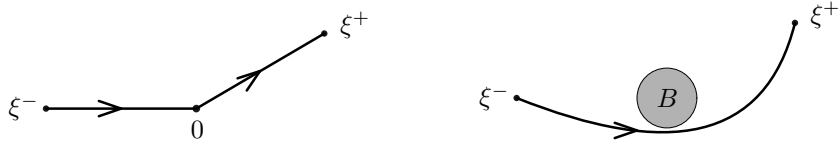

\begin{center}
\begin{mac}
\begin{tabular}{ccc}
\begin{texdraw}
\drawdim cm  \setunitscale 1
%%%%%%%%%%%%%%%%
\linewd 0.03
\fcir f:0.1 r:0.05
\linewd 0.01
%\fcir f:0.9 r:2
%\larc r:2 sd:0 ed:360
\linewd 0.03
\move (-2 0)
\fcir f:0.1 r:0.04
\lvec (0 0)
\lvec (1.7 1)
\fcir f:0.1 r:0.04
%%%%%%%%%%%%%%%%
%% frecce
%%%%%%%%%%%%%%%%
\move (-1.1 0) \arrowheadtype t:V \arrowheadsize l:0.3 w:0.2 \avec (-1 0)
\move (0.81 0.475) \arrowheadtype t:V \arrowheadsize l:0.3 w:0.2 \avec (0.85 0.5)

%%%%%%%%%%%%%%%%%%
%% scritte
%%%%%%%%%%%%%%%%%%
\textref h:C v:C
\htext (-2.3 0) {$\xi^-$}
\htext (2.1 1.1) {$\xi^+$}
\htext (0 -0.3) {$0$}
\end{texdraw}

& \hspace{1cm} &

\begin{texdraw}
\drawdim cm  \setunitscale 1
%%%%%%%%%%%%%%%%
\linewd 0.01
%\fcir f:0.9 r:2
\fcir f:0.7 r:.4
%\larc r:2 sd:0 ed:360
\larc r:.4 sd:0 ed:360
\linewd 0.03
\move (-2 0)
\fcir f:0.1 r:0.04
\move (1.7 1)
\fcir f:0.1 r:0.04
%%%%%%%%%%%%%%%
%% traiettoria
%%%%%%%%%%%%%%%
\move (-2 0)\clvec(0 -0.8)(1.3 -0.6)(1.7 1)
\move (-0.43 -0.428) \arrowheadtype t:V \arrowheadsize l:0.3 w:0.2 \avec (-0.4 -0.43)
%%%%%%%%%%%%%%%%%%
%% scritte
%%%%%%%%%%%%%%%%%%
\textref h:C v:C
\htext (0 0) {$B$}
\htext (-2.3 0) {$\xi^-$}
\htext (2.1 1.1) {$\xi^+$}
\end{texdraw}
\end{tabular}
\end{mac}
\end{center}
\caption{ at left, $c(V)$ is achieved by a double-homothetic motion (Proposition \ref{propo:alter}, case (1)); at right  $c(V)$ is achieved by a non-collision trajectory (Proposition \ref{propo:alter}, case (2)). When the second situation occurs, there exists a ball $B$, centered at the origin, such that any trajectory that achieves $c(V)$ does not intersect $B$.\label{fig:intro}}
\end{figure}
Following the previous proposition, we distinguish potentials with ``inner'' minimizers
(i.e. minimizers which pass through the origin) from potential with ``outer'' ones:
\begin{equation*}%\label{eq:in}
	\setin := \left\{ V \in \Poh : c(V)=4\sqrt{2V_{\min}}/(2-\alpha)  \right\},
\end{equation*}
\begin{equation*}%\label{eq:out}
	\setout := \left\{ V \in \Poh : c(V)<4\sqrt{2V_{\min}}/(2-\alpha)  \right\}.
\end{equation*}
It is easy to see that these two sets are disjoint and their union is the whole $\Poh$;
moreover we will show that the first one is closed while the second is open.
We are interested in their common boundary, that is
\begin{equation*}%\label{eq:Pi}
\Pi :=\partial{\setin} \cap \partial{\setout}.
\end{equation*}
The separating property of the common boundary is underlined by the following lemma.
\begin{lemma}\label{lem:struct}
There exists an open nonempty set $\Sigma \subset \Seh$, and a continuous function
             $\bar \alpha : \Sigma \to (0,2)$ such that
             \[
               \Pi =\left\{(V,\bar \alpha(V)):\,V \in \Sigma\right\}.
             \]
\end{lemma}
\noindent Furthermore, we will provide explicit criteria in order to establish whether a potential $V\in\Seh$
belongs to the domain of the function $\bar\alpha$.

The main result of this paper states that the above graph coincides with the set of potentials
admitting parabolic Morse minimizers.
\begin{theorem}\label{MainTheorem}
$V \in \Poh$ admits a parabolic Morse minimizer if and only if $V \in \Pi$.
\end{theorem}
\begin{remark}\label{rem:weakkam}
Of course, due to the invariance by homotheticity of the problem, such Morse minimizing parabolic trajectories always come in one-parameter families and give rise to a 2-dimensional Lagrangian submanifold having boundary corresponding to the two homothetic solutions (see Figure \ref{fig: lampadine}). In the planar case, minimal orbits can be considered in a given homotopy class of paths with values in $\RR^2\setminus\{0\}$, and the same
can be done in any non simply connected target. With this variant in mind, it is meaningful to have equal ingoing and outgoing asymptotic directions. In this case we find the aforementioned lamination of the configuration space giving rise to a solution of the Hamilton-Jacobi equation associated with \eqref{eq:dynsys}, which is $\mathcal C^1$ on the double  covering of $\RR^2\setminus\{0\}$.
\end{remark}
\begin{figure}[!t]
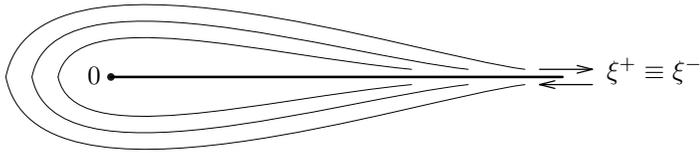

\begin{center}
\begin{mac}
\begin{texdraw}
\drawdim cm  \setunitscale 1
{
%%%%%%%%%%%%%%%%
\linewd 0.03
\fcir f:0.1 r:0.05
\lvec (6 0)
\linewd 0.01
%% piccola
\move (-.7 0) \clvec(-.5 1)(2 .3)(4 .1)
\move (-.7 0) \clvec(-.5 -1)(2 -.3)(4 -.1)
%% media
\move (-1.05 0) \clvec(-.75 1.5)(3 .3)(4.75 .1)
\move (-1.05 0) \clvec(-.75 -1.5)(3 -.3)(4.75 -.1)
%% grande
\move (-1.4 0) \clvec(-1 2)(4 .3)(5.5 .1)
\move (-1.4 0) \clvec(-1 -2)(4 -.3)(5.5 -.1)
%% frecce
\linewd 0.02
\move (5.7 .1)  \arrowheadtype t:V \arrowheadsize l:0.2 w:0.1 \avec (6.4 .1)
\move (6.4 -.1) \arrowheadtype t:V \arrowheadsize l:0.2 w:0.1 \avec (5.7 -.1)
%% scritte
\htext (6.6 -.1) {$\xi^+\equiv\xi^-$}
\htext (-.3 -.1) {$0$}
\move (0 -.5)
}
\end{texdraw}
\end{center}
\end{mac}
\caption{ a one-parameter family of planar Morse minimizing parabolic
trajectories with the same asymptotic direction at $+\infty$ and $-\infty$
and nontrivial topology. \label{fig: lampadine}}
\end{figure}
In the literature, minimal parabolic trajectories have been studied in connection
with the absence of collisions for fixed-endpoints minimizers.
More precisely, as remarked by Luz and Maderna in \cite{LuzMad2011},
the property to be collisionless for all Bolza minimizers implies the absence of
parabolic trajectories which are Morse minimal.
In particular, as they point out, this is the case for the $n$-body problem,
when no topological constraints are imposed. On the contrary,  minimal parabolic arcs (i.e.,
defined only on the half line) exist for every starting configuration,
as proved by Maderna and Venturelli in \cite{MadVen2009}. Up to our knowledge,
the present paper is the first with positive results about the existence of globally defined
parabolic minimizers.

The paper is structured as follows.
In Section \ref{sec:planar}, we give an account of the planar case $d=2$, relating minimal
parabolic trajectories with the aforementioned minima connections in Devaney's work.
Next, to construct global-in-time Morse minimizers in higher dimension, we first consider
problems on bounded intervals (Sections \ref{sec:Bolza}, \ref{sec:level_estimates}), and
then pass to the limit (Section \ref{sec:Morse}). This procedure may fail for two main reasons: sequences
of approximating trajectories may either converge to the singularity, or escape to infinity.
This naturally leads to introduce some constraint in our construction, and to define
\emph{constrained Morse minimizers}, satisfying
$\min_{t} |x(t)|=\ep$ (see Definition \ref{defi:constr_Morse_min}). The study of the interaction of
such minimizers with the constraint leads to the definition of the \emph{position-jump} $\pjump$ and of the
\emph{velocity-jump} $\vjump$ of a trajectory, see Figure \ref{fig:figureJUMPS}. Under this perspective,
the crucial fact is that such quantities do not depend on the minimizer, but only on the potential (they
are indeed related to the corresponding apsidal angle).
In Section \ref{sec:ParMin} we give the full details of the relations between parabolic minimizers,
constrained Morse minimizers, position- and velocity-jumps, and the separating interface $\Pi$,
obtaining as a byproduct the
proof of Theorem \ref{MainTheorem}. Finally, for the reader's convenience, in the appendices we collect
the proof of some rather known results for which we could not find an appropriate reference.

%==============================
\subsubsection*{Notations}\label{subsec:prel}
%==============================
%
Throughout the paper we will often use {polar coordinates},
that corresponds to writing $x=rs$, where
\[
r=|x|\geq0\quad\text{ and }\quad s=\frac{x}{|x|}\in\sphere.
\]
With this notation equations \eqref{eq:dynsys} and \eqref{eq:dynsys2} read as
\begin{equation}\label{eq:ELrs}
\begin{cases}
\ddot r & = -\dfrac{\alpha}{2}\dfrac{\dot r}{r}\dot r + \dfrac{2-\alpha}{2} |\dot s|^2 r\\ \vspace{-.3cm}\\
\ddot s & = -2\dfrac{\dot r}{r}\dot s + \dfrac{\nabla_T V(s)}{r^{2+\alpha}} -  |\dot s|^2 s,
\end{cases}
\qquad \qquad
\frac12 \dot r^2+ \frac12 r^2 |\dot s|^2 = \frac{V(s)}{r^\alpha};
\end{equation}
here $\nabla_T V(s)$ denotes the projection of $\nabla V(s)$
on the tangent space $T_s\sphere$:
\begin{equation}\label{eq:inpiu}
\nabla_T V(s) =
\nabla V(s)-\left(\nabla V(s)\cdot s\right)s = \nabla V(s) + \alpha V(s)s.
\end{equation}
Finally we will denote with $V_{\min}$ and $V_{\max}$ the extrema of $V|_{\sphere}$,
with
\[
\alpha_*:= \frac{2-\alpha}{2}>0,
\]
and with $C$ any (positive) constant we do not need to specify.

%==============================
\section{The Planar Case}\label{sec:planar}
%==============================
%
As a guideline for our higher dimensional studies, in this section
we consider the planar anisotropic Kepler problem.
Indeed, following Devaney \cite{DevInvMath1978,DevProgMath1981}, when dealing with zero-energy solutions
this problem is equivalent, after some suitable change of variables, to a bi-dimensional autonomous
dynamical system, for which explicit calculations can be carried out.

We briefly sketch Devaney's procedure. Introducing the standard polar coordinates $(r,\vartheta)$,
the potential $V:\cerchio \to \RR$ is a $2\pi$-periodic function in $\vt$;
for a clear-cut notation we define:
\[
U(\vt) := V(\cos \vt, \sin \vt) \geq U_{\min} >0, \qquad \vt \in \RR,
\]
and we then deal with the extended $-\alpha$-homogeneous potential
\[
V(r\cos \vt,r\sin\vt) = \frac{U(\vt)}{r^\alpha}, \qquad \vt \in \RR, \, r>0.
\]
Introducing the Cartesian coordinates $q_1=r\cos \vt$, $q_2=r\sin\vt$
and the momentum vector $(p_1,p_2) = (\dot q_1,\dot q_2)$, we write
%equations \eqref{eq:dynsys} and \eqref{eq:dynsys2} as
%\[
%\begin{cases}
%\dot q_1 = p_1 \\
%\dot q_2 = p_2 \\
%\dot p_1 = {\partial_{q_1}}\left( {r^{-\alpha}}{U(\vt)} \right)
%         = {r^{-\alpha-2}}\left( -U'(\vt)q_2 -\alpha U(\vt)q_1 \right) \\
%\dot p_2 = {\partial_{q_2}}\left( {r^{-\alpha}}{U(\vt)} \right)
%         = {r^{-\alpha-2}}\left( U'(\vt)q_1 -\alpha U(\vt)q_2 \right),
%\end{cases}
%\]
%and
%\[
%\frac12 \left( p_1^2 + p_2^2\right) = \frac{U(\vt)}{r^\alpha}.
%\]
%Since $U(\vartheta)\geq U_{\min}>0$, we have that $|p|\neq0$.
%As a consequence, for every solution of the previous dynamical system we can find smooth functions
%$z >0$ and $\vp \in \RR$ in such a way that
$p_1=r^{-\alpha/2}z\cos \vp$, $p_2=r^{-\alpha/2}z\sin\vp$, for suitable smooth functions
$z >0$ and $\vp \in \RR$.
Under these notations, equations \eqref{eq:dynsys2} and \eqref{eq:dynsys} become
\[
z = \sqrt{2U(\vt)}
\]
and
\[
\begin{cases}
\dot r   %= r^{-\alpha/2}z\left( \cos \vt \cos \vp + \sin \vt \sin \vp \right)
         = r^{-\alpha/2}z \cos (\vp-\vt) \smallskip\\
\dot \vt %= r^{-1-\alpha/2}z\left( \cos \vt \sin \vp - \sin \vt \cos \vp \right)
         = r^{-1-\alpha/2}z\sin (\vp-\vt) \smallskip\\
\dot z   = r^{-1-\alpha/2}U'(\vt)\sin(\vp-\vt) \smallskip\\
\dot \vp = \frac{1}{z} r^{-1-\alpha/2} \left[ U'(\vt)\cos(\vp-\vt) +\alpha U(\vt) \sin(\vp-\vt)\right].
\end{cases}
\]
The singularity at $r=0$ can be removed by a change of time scale
\[
\frac{\dt}{\dtau} = zr^{1+\alpha/2}.
\]
This allows to rewrite the system as (here ``\,$'$\,'' denotes the derivative with respect to $\tau$)
\begin{equation*}%\label{eq:syscomplete}
\begin{cases}
r'   = 2r U(\vt) \cos (\vp-\vt) \\
z'   = z U'(\vt)\sin(\vp-\vt) \\
\vt' = 2 U(\vt) \sin (\vp-\vt) \\
\vp' = U'(\vt)\cos(\vp-\vt) +\alpha U(\vt) \sin(\vp-\vt),
\end{cases}
\end{equation*}
which solutions are globally defined in $\tau$.
Let us concentrate on the first equation: on one hand we have that $r$ can never vanish;
on the other hand, to ensure that $r$ is
unbounded both in the past and in the future, it is sufficient to check that
\begin{equation}\label{eq:cond_cos}
\text{$\cos(\vp-\vt)$ is bounded away from zero and positive (resp. negative)}
\end{equation}
as $\tau \to +\infty$ (resp. $-\infty$).
Furthermore this also implies that $t \to \pm\infty$ as $\tau$ does.
Keeping in mind the above condition, the study of (not necessarily minimal)
parabolic orbits reduces to the one of the planar system
\begin{equation}\label{eq:dydthetaphi}
\begin{cases}
\vt' = 2 U(\vt) \sin (\vp-\vt) \\
\vp' = U'(\vt)\cos(\vp-\vt) +\alpha U(\vt) \sin(\vp-\vt).
\end{cases}
\end{equation}
To start with, let us take into account the situation when the potential $U$ is {\em isotropic},
for instance $U(\vartheta)\equiv 1$; in this case, every $\vartheta$ is a minimal central configuration, and
the dynamical system above reads
\[
\vp' = \frac{\alpha}{2} \vt' = \alpha \sin(\vp-\vt),
\]
which critical points satisfy $\varphi = \vartheta +k\pi$, $k \in \ZZ$.
Furthermore, trajectories lie on the bundle $\varphi = (\alpha/2) \vartheta + C$, $C \in \RR$.
Recalling condition \eqref{eq:cond_cos}, we infer that parabolic solutions coincide
with heteroclinic connections departing from points on
$\varphi = \vartheta +(2k+1)\pi$ and ending on $\varphi = \vartheta +2k\pi$, for some $k\in\ZZ$.
For instance, when $k=0$, we obtain heteroclinics connecting $(\vartheta^*,\vartheta^*+\pi)$ to
$(2\pi/(2-\alpha)+\vartheta^*,2\pi/(2-\alpha)+\vartheta^*)$, for some $\vartheta^* \in \RR$.
Going back to the original dynamical system, this implies that parabolic motions exists
only when the angle between the ingoing and outgoing asymptotic directions
is $2\pi/(2-\alpha)$.
Dealing with the Kepler problem ($\alpha =1$) this angle is $2\pi$, hence
the heteroclinic between $(\vartheta^*,\vartheta^*+\pi)$ and
$(2\pi+\vartheta^*,2\pi+\vartheta^*)$
describes a parabola whose axis form an angle $\vartheta^*$ with the horizontal line.
It is worthwhile noticing that, despite connecting minimal configurations, these parabolic trajectories
are not globally minimal in the sense of Definition \ref{defi:Morse_min}
(it can be shown that they are minimal in their homotopy class).

\begin{figure}[!t]
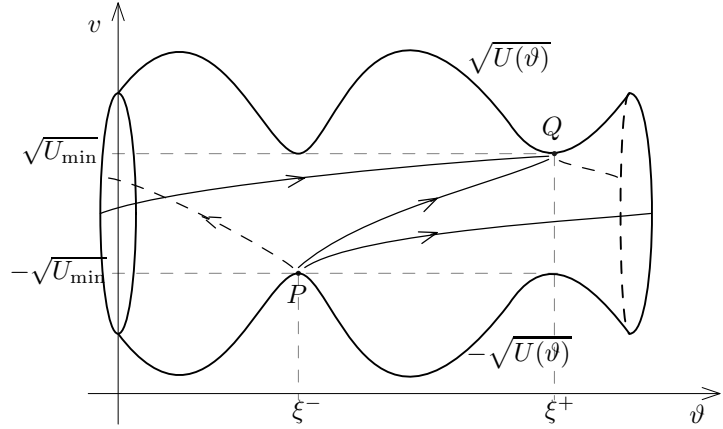

\begin{center}
\begin{mac}
\begin{texdraw}
\drawdim cm  \setunitscale .8
%% linee tratteggiate orizzionatali-verticali
\linewd 0.01
\setgray 0.5 \lpatt (0.2 0.2)
\move (-0.1 -1) \lvec (7.3 -1)
\move (-0.1 1) \lvec (7.3 1)
\move (3 -1) \lvec (3 -3.1)
\move (7.25 1) \lvec (7.25 -3.1)
\lpatt()
\setgray 0
\htext (-1.8 -1.2) {$-\sqrt{U_{\min}}$}
\htext (-1.6 0.8) {$\sqrt{U_{\min}}$}
\htext (2.9 -3.5) {$\xi^-$}
\htext (7.1 -3.5) {$\xi^+$}
%% grafico sotto %%
\linewd 0.03
\move (0 -2)
\clvec(1.5 -4)(2.4 -1)(3 -1)
\fcir f:0.1 r:.04
\clvec(3.6 -1)(4.2 -4.5)(6.5 -1.5)
\clvec(7 -.75)(7.60 -0.8)(8.5 -2)
\htext (5.8 -2.6) {$-\sqrt{U(\vartheta)}$}
\htext (2.8 -1.5) {$P$}
%% grafico sopra %%
\move (0 2)
\clvec(1.5 4)(2.4 1)(3 1)
\clvec(3.6 1)(4.2 4.5)(6.5 1.5)
\clvec(7 .75)(7.60 0.8)(8.5 2)
\htext (5.8 2.3) {$\sqrt{U(\vartheta)}$}
\move (7.25 1) \fcir f:0.1 r:.04
\htext (7.05 1.2) {$Q$}
%% ellissi laterali per 3d %%
\move (0 0)
\lellip rx:.3 ry:2
\move (8.5 -2)
\clvec(9 -2)(9 2)(8.5 2)
\lpatt (0.2 0.2)
\clvec(8.3 2)(8.3 -2)(8.5 -2)
\lpatt ()
%% traiettoie da P %%
\linewd 0.02
\move (3 -.9)
\clvec(4 -.0)(6.25 0.5)(7.15 .9)
\move (5.2 0.2)
\arrowheadtype t:V \arrowheadsize l:0.3 w:0.2 \avec (5.3 0.25)
\move (3.1 -.9)
\clvec(4 -.3)(8 -0.1)(8.85 0)
\move (5.2 -.35)
\arrowheadtype t:V \arrowheadsize l:0.3 w:0.2 \avec (5.3 -.32)
\move (-0.3 0)
\clvec(1 .5)(6 0.9)(7.15 .95)
\move (3 .575)
\arrowheadtype t:V \arrowheadsize l:0.3 w:0.2 \avec (3.1 .59)
\move (2.9 -.9)
\lpatt (0.2 0.2) \clvec(2.7 -.7)(.4 .6)(-.3 .6)
\move (8.3 .6) \clvec(8.1 .7)(7.4 .8)(7.3 .9)
\move (1.5 -0.1)
\arrowheadtype t:V \arrowheadsize l:0.3 w:0.2 \avec (1.4 -0.05)
\lpatt ()
%% frecce-assi %%
\linewd 0.01
\move (0 -3.5) \arrowheadtype t:V \arrowheadsize l:0.3 w:0.2 \avec (0 3.5)
\htext (-0.5 3) {$v$}
\move (-0.5 -3) \arrowheadtype t:V \arrowheadsize l:0.3 w:0.2 \avec (10 -3)
\htext (9.5 -3.5) {$\vartheta$}
\end{texdraw}
\end{mac}
\end{center}
\caption{ a saddle connection in the phase plane of system \eqref{eq:dydthetaphi} corresponds
to a heteroclinic connection between $P$ and $Q$.\label{fig:tubo}}
\end{figure}

If $U$ is not constant,
stationary points of \eqref{eq:dydthetaphi} are $(\vt^*,\vp^*)$ such that $\sin(\vp^*-\vt^*)=0$, and $U'(\vt^*)=0$.
By linearizing it is easy to see that non-degenerate minima (resp. maxima) $\vt^*$ for $U$ correspond
to critical points $(\vt^*,\vp^*)$ which are saddles (resp. sinks/sources).
Accordingly, taking into account condition \eqref{eq:cond_cos},
let us assume that the system admits a pair of saddles
$(\vt_1^*,\vp_1^*)$, $(\vt_2^*,\vp_2^*)$, such that
$\cos (\vp_1^*-\vt_1^*)=-1$ and $\cos (\vp_2^*-\vt_2^*)=1$.

Let us define the function
\[
v(\tau) = \sqrt{U(\vartheta(\tau))} \cos\left( \varphi(\tau)-\vartheta(\tau) \right),
\]
which satisfies
\[
-\sqrt{U(\vartheta(\tau))} \leq v(\tau) \leq \sqrt{U(\vartheta(\tau))}.
\]
By direct computations, we obtain that $v$ is  non-decreasing on the solutions of \eqref{eq:dydthetaphi}.
Let now $\xi^{-} = \vt_1^*$ and $\xi^{+} = \vt_2^*$ be defined as above.
A parabolic trajectory between $\xi^{\pm}$ projects, on the plane $(\vt,v)$, on
an increasing graph connecting $P=(\xi^-,-\sqrt{U_{\min}})$ and $Q=(\xi^+,\sqrt{U_{\min}})$  (we refer to Figure
\ref{fig:tubo}).

\begin{figure}[!t]
\centering
\includegraphics[width=5.5cm]{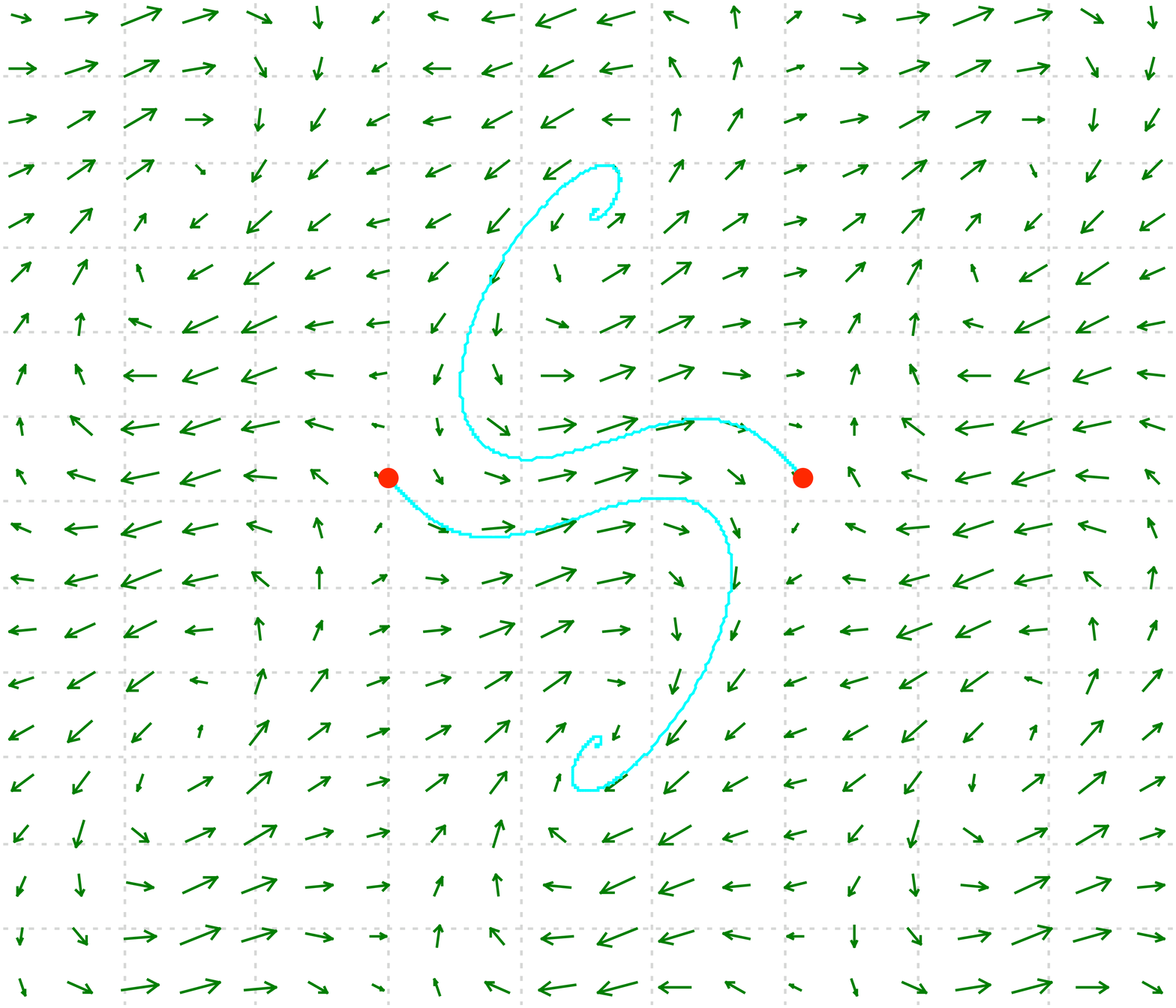}
\qquad
\includegraphics[width=5.5cm]{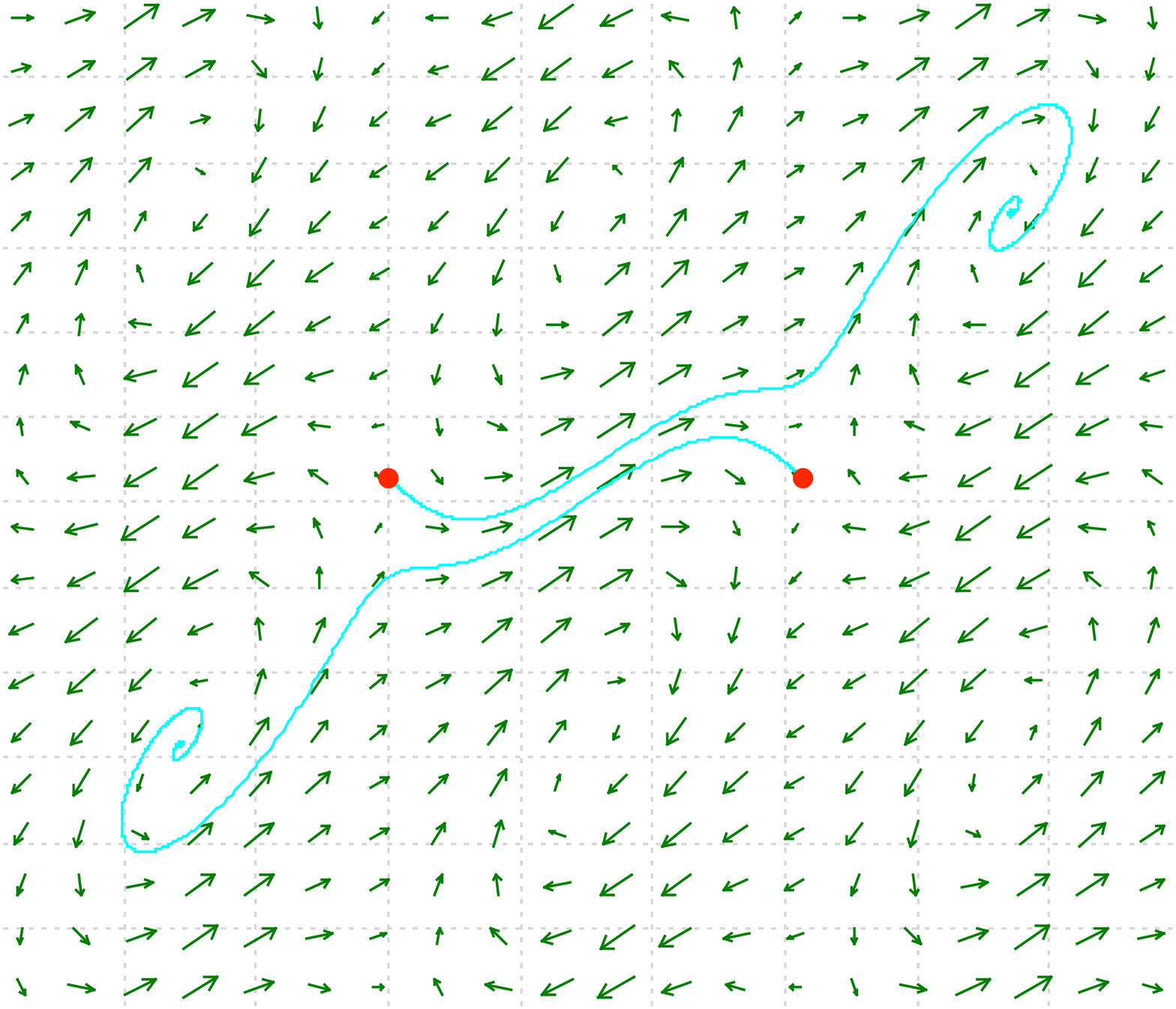}
\caption{the two pictures represent the phase portrait of the dynamical system \eqref{eq:dydthetaphi}
with $U(\vt) = 2-\cos (2\vt)$, when $\alpha = 0.5$ (at left) or $\alpha = 1$ (at right).
We focus our attention on the saddles $(0,\pi)$ and $(\pi,\pi)$ (that satisfy condition \eqref{eq:cond_cos}):
from the mutual positions of the
heteroclinic departing from $(0,\pi)$ and the one ending in $(\pi,\pi)$ we deduce that the
two vector fields are not topologically equivalent.
Using standard arguments in the theory of structural stability
(e.g. Theorem 13.6 in \cite{HK}), we infer the existence, for some $\bar\alpha\in(0.5,1)$, of a saddle connection
between $(0,\pi)$ and $(\pi,\pi)$.
\label{fig:pplane}}
\end{figure}

Generically (see also Theorem 4.10 in \cite{DevInvMath1978}),
the unstable manifold at $P$ falls directly into a sink,
while the stable manifold at $Q$ emanates from a source:
we claim that, for some values of the parameter $\alpha$, such two points can be connected.
Let us focus on the heteroclinic from $P$.
Since the derivative of $v$ can be written as
\[
v'(\tau) = (2-\alpha)\sqrt{U(\vartheta(\tau))} \left[U(\vartheta(\tau))-v^2(\tau)\right],
\]
$v$ strictly increases whenever $v(\tau) \in (-\sqrt{U(\tau)},\sqrt{U(\tau)})$;
in such monotonicity intervals we have that $\sin \left( \varphi(\tau)-\vartheta(\tau) \right) \neq 0$, hence also
$\vartheta$ is strictly monotone.
As a consequence we can read
$v$ as a function of $\vartheta$ (inverting $\vartheta(\tau)$)).
Our aim is to show that, for some $\alpha \in (0,2)$ and some $k \in \NN$,
there exists a solution of the dynamical system such that
\[
v(\xi^-) = -\sqrt{U_{\min}} \quad \text{and} \quad v(\xi^+ + 2k\pi) = \sqrt{U_{\min}}.
\]
Since
\begin{multline*}
\frac{\d v}{\dtheta} = v'(\tau) \frac{\dtau}{\dtheta}
                     = \frac{2-\alpha}{2}\frac{\sqrt{U(\vartheta(\tau))}}{U(\vt(\tau))}
                       \frac{U(\vartheta(\tau))-v^2(\tau)}{\sin (\vp-\vt)} \\
                     = \frac{2-\alpha}{2}\sqrt{U(\vartheta)-v^2(\tau(\vartheta))},
\end{multline*}
integrating on $\vartheta \in [\xi^-,\hat \vartheta]$ and  $v \in [-\sqrt{U_{\min}},\sqrt{U_{\min}}]$, we obtain
on one hand
\[
\hat \vartheta - \xi^-
\leq \frac{2}{2-\alpha}\int_{-\sqrt{U_{\min}}}^{\sqrt{U_{\min}}} \frac{\d v}{\sqrt{U_{\min}-v^2}}
=    \frac{2\pi}{2-\alpha}
\]
and on the other hand
\[
\hat \vartheta - \xi^-
\geq \frac{2}{2-\alpha}\int_{-\sqrt{U_{\min}}}^{\sqrt{U_{\min}}} \frac{\d v}{\sqrt{U_{\max}-v^2}}
=    \frac{4}{2-\alpha}\arcsin \sqrt\frac{U_{\min}}{U_{\max}}.
\]
From the first inequality we deduce that, as $\alpha$ becomes very small,
$\hat \vartheta$ does not exceed $\xi^- + \pi$;
from the second one we infer that, as $\alpha$ tends to 2,
$\hat \vartheta$ diverges to $+\infty$.
It is possible to conclude that, for any (large) $k$
there exists $\bar \alpha_k$
such that $v(\xi^-) = -\sqrt{U_{\min}}$ and $v(\xi^+ + 2k\pi) = \sqrt{U_{\min}}$,
see Figure \ref{fig:pplane}.
More results in this direction are contained in \cite{BTVpreparation}.

%======================================
\section{Bolza Minimizers}\label{sec:Bolza}
%======================================
Now we turn to the general case of dimension $d$.
In this section we investigate constrained fixed-endpoints problems for
lagrangians with a potential $V=(V,\alpha)$, under the assumptions that $V$ is
$-\alpha$-homogeneous and that $V|_{\sphere}$ is positive and smooth.
In particular all the results will hold if $V \in \Poh$, even though here the assumptions
about $\xi^{\pm}$ do not play any role.

Let us fix $\ve>0$, $x_1, x_2 \in \RR^d \setminus B_{\ve }(0)$, and $T>0$.
We introduce the sets of constrained paths
\[
\GRe{T}{\ve }{x_1}{x_2} :=
\left\{ x\in H^1(-T,T) :\, x(-T)=x_1, x(T)=x_2, \min_{t\in[-T,T]} |x(t)| = \ve  \right\}
\]
and their union
\[
\GReU{\ve }{x_1}{x_2} = \GReU{\ve }{x_1}{x_2}({x_1},{x_2},{\ve }) := \bigcup_{T>0} \GRe{T}{\ve }{x_1}{x_2}.
\]
This section is devoted to study the Bolza minimization problem
\begin{equation}\label{eq:minA}
m = m(x_1,x_2,\ve) := \inf_{x \in \GReU{\ve}{x_1}{x_2}} \action(x).
\end{equation}
Let us remark that, for a unified treatment, we let $x_1$ and/or $x_2$ belong to the constraint $|x|=\ve$.
To avoid degenerate situations we suppose that
$|x_1|=|x_2|=\ve$ implies $x_1 \neq x_2$, excluding the trivial case.

As we noticed in the introduction, we minimize with respect to both trajectory $x$ and time length $T$.
The reason is that such procedure will provide zero-energy motions (see Appendix \ref{app:maup}).
To exploit this property we define, for any $y \in H^1(-1,1)$, the \emph{Maupertuis' functional}
\[
J(y) = J([-1,1];y):= \int_{-1}^1 \frac12 |\dot y(t)|^2\,\dt \int_{-1}^1 V(y(t))\,\dt.
\]
%
%\begin{remark}\label{rem:add}
%Let us assume that $\bar x$, defined on $(-T,T)$, achieves $m$,
%and let $(a,b) \subset (-T,T)$.
%From the definition of $\GReU{}{}{}$
%a remarkable property on the restriction $\hat x = \bar x|_{(a,b)}$ follows:
%indeed (a suitable time translation of) $\hat x$ minimizes $\action$ on the set
%\[
%\bigcup_{T>0} \left\{ x\in H^1(-T,T) :\, x(-T)=\bar x(a), x(T)=\bar x(b)\right\},
%\]
%with the further constraint that
%$\min |x(t)| = \ve$ if $\min|\hat x(t)|=\ve$, and that
%$\min |x(t)| > \ve$ if $\min|\hat x(t)|>\ve$.
%Roughly speaking, this means that we are allowed to make variations
%that also change the definition time interval.
%\end{remark}
%
%For an easier notation we will often denote,
%\[
%\kin{x}{-T}{T} := \int_{-T}^T \frac12 |\dot x(t)|^2\,\dt, \qquad \pot{x}{-T}{T}:=\int_{-T}^T V(x(t))\,\dt
%\]
%and we will refer to them respectively as to the \emph{kinetic}  and \emph{potential integral} of $x$.
%
\begin{lemma}\label{lem:legameAJ}
If $\bar x \in \GRe{T}{\ve}{x_1}{x_2}$ achieves $m$,
then $\bar y(t) := \bar x(Tt)$, $t \in [-1,1]$, achieves
\begin{equation*}\label{eq:minJ}
 \inf_{y \in \GRe{1}{\ve}{x_1}{x_2}} J(y).
\end{equation*}
On the other hand, if $\bar y \in \GRe{1}{\ve}{x_1}{x_2}$ achieves the infimum above, then
$\bar x(\tau) := \bar y(\tau/\bar T )$, $\tau \in [-\bar T,\bar T]$ achieves $m$ where
\begin{equation*}\label{eq:barT}
\bar T := \left(\frac{\int_{-1}^1 |\dot {\bar y}|^2}{2\int_{-1}^1 V(\bar y)}\right)^{1/2}.
\end{equation*}
%In both cases we have
%\[
%\int_{-T}^T \frac12 |\dot {\bar x}(t)|^2\,\dt = \int_{-T}^T V({\bar x}(t))\,\dt, \quad \text{and}\quad
%\action(\bar x) = 2\sqrt{J(\bar y)}.
%\]
\end{lemma}
\begin{proof}
For any $T>0$, there exists a one-to-one correspondence between the sets
$H^1(-1,1)$ and $H^1(-T,T)$: $x \in H^1(-T,T)$ if and only if $y_x(t):= x(Tt) \in H^1(-1,1)$.
Taking into account this fact, the lemma follows by arguing as in the proof of Lemma \ref{lem:maup}.
%Second, since by a simple time-scaling we have
%\[
%\action(x) = \kin{x}{-T}{T} + \pot{x}{-T}{T} %\int_{-T}^{T} \frac12 |\dot x|^2 + V(x)
%           = \frac1T\,\kin{y_x}{-1}{1} + T\,\pot{y_x}{-1}{1} %\int_{-1}^{1} \frac{1}{2} |\dot y_x|^2 + TV(y_x)
%\]
%the following chain of equalities hold
%\[
%\begin{split}
%\inf_{\GReU{\ve}{x_1}{x_2}} \action(x)
%& = \inf_{T>0}
%    \left[ \inf_{y \in \GRe{1}{\ve}{x_1}{x_2}} \left(\frac1T\,\kin{y}{-1}{1} + T\,\pot{y}{-1}{1}\right) \right] \\
%& = \inf_{y \in \GRe{1}{\ve}{x_1}{x_2}}
%    \left[ \inf_{T>0} \left(\frac1T\,\kin{y}{-1}{1} + T\,\pot{y}{-1}{1}\right) \right] \\
%& = \inf_{y \in \GRe{1}{\ve}{x_1}{x_2}} 2 \sqrt{\kin{y}{-1}{1}} \, \sqrt{\pot{y}{-1}{1}},
%\end{split}
%\]
%indeed, fixed $a,b >0$, the minimal value of the function $f(T):=\frac{a}{T}+bT$
%is achieved when $T = \sqrt{a/b}$. Finally, recalling that $\bar T>0$ (see equation \eqref{eq:initial}),
%this implies
%\[
%\kin{\bar x}{}{} = \frac{1}{\bar T}\kin{\bar y}{}{}=
%\sqrt{\kin{\bar y}{-1}{1}} \, \sqrt{\pot{\bar y}{-1}{1}}=
%\bar T \pot{\bar y}{}{} = \pot{\bar x}{}{}.\qedhere
%\]
\end{proof}

The reformulation of problem \eqref{eq:minA} in terms of $J$
allows to easily prove the existence of a minimizer.

\begin{lemma}\label{lem:Boltza_exists}
The infimum of $J$ on $\GRe{1}{\ve }{x_1}{x_2}$ is achieved.
\end{lemma}
\begin{proof}
Let $\suc{y}{n}\subset \GRe{1}{\ve }{x_1}{x_2}$ be a minimizing sequence.
We claim that $\int_{-1}^{1}V({y_n})$ is bounded away from zero. If this is true
then the lemma follows in a standard way: indeed, as a consequence,
$\suc{y}{n}$ is uniformly bounded and hence weakly convergent in $H^1$,
$\GRe{1}{\ve }{x_1}{x_2}$ is weakly closed,
and $J$ is weakly lower semi-continuous.
To prove the claim, let us assume by contradiction that $\delta_n:=\int_{-1}^{1}V({y_n})\to 0$;
then there exists $t_n' \in [-1,1]$
such that $V(y_n(t_n'))<\delta_n$ and then, by homogeneity,
$|y_n(t_n')| > \left(V_{\min}/\delta_n\right)^{1/\alpha}$.
By H\"older's inequality we have, as  $n$ becomes large,
\[
\begin{split}
 \int_{-1}^{1}\frac12|\dot y_n|^2
 \geq \frac14 \left( \int_{-1}^1|\dot y_n| \right)^2
 & \geq \frac14 \left| y_n(t_n')-x_1\right|^2
 \geq \frac14 \left(|y_n(t_n')| - |x_1| \right)^2\\
 & \geq \frac14 \left( \frac{V_{\min}^{1/\alpha}}{\delta_n^{1/\alpha}} - |x_1| \right)^2
 \geq C\delta_n^{-2/\alpha}.
\end{split}
\]
Hence $J(y_n) \geq C\delta_n^{-(2-\alpha)/\alpha}$ and, since $\alpha \in (0,2)$,
this contradicts the fact that $\suc{y}{n}$ is a minimizing sequence.
\end{proof}

Recalling Lemmata \ref{lem:legameAJ} and
\ref{lem:generale} we have the following result.

\begin{corollary}\label{coro:energy}
$m$ is achieved in $\Gamma$.
For every $\bar x = \bar r \bar s$ minimizer, the function $|\dot{\bar x}|^2= \dot{\bar r}^2 + \bar r^2
|\dot{\bar s}|^2 $ is continuous.
In particular \eqref{eq:dynsys2} holds for every $t$.
\end{corollary}

%\begin{lemma}\label{lem:energia}
%Let us assume that $\bar x\in \GRe{\bar T}{\ve}{x_1}{x_2}$ achieves $m$
%and that $|\bar x(t)|>\ve$, for $t \in (a,b) \subset (-\bar T,\bar T)$.
%Then
%\[
%\ddot{\bar x}(t) = \nabla V(\bar x(t))
%\quad \text{and} \quad
%\frac12 |\dot{\bar x}(t)| = V(\bar x(t)), \quad \text{for every } t \in (a,b).
%\]
%\end{lemma}
%
%\begin{proof}
%Since $\bar x|_{(a,b)}$ does not interfere with the constraint,
%the validity of the Euler-Lagrange equation follows in a standard way.
%As a consequence the conservation of energy
%\[
%\frac12 |\dot{\bar x}(t)|^2 = V(\bar x(t))+h, \quad \text{for every } t \in (a,b),
%\]
%holds, for some $h \in \RR$.
%Finally, using Remark \ref{rem:add} and reasoning as in the proof of Lemma \ref{lem:legameAJ},
%we obtain that
%\[
%\int_a^b\frac12 |\dot{\bar x}(t)|\dt = \int_a^b V(\bar x(t))\dt,
%\]
%hence $h=0$.
%\end{proof}

\begin{corollary}\label{cor:convex}
Let $\bar x\in\Gamma_T$ achieve $m$.
Then there exist $t_* \leq t_{**}$ such that
\begin{itemize}
\item $\bar r(t) = \ve$ if and only if $t \in [t_{*},t_{**}]$;
\item for every $t \in (-T,t_*)$ we have $\dot{\bar r}(t)< 0$ (and \eqref{eq:dynsys}
holds);
\item for every $t \in (t_{**},T)$ we have $\dot{\bar r}(t) > 0$ (and \eqref{eq:dynsys}
holds).
\end{itemize}
\end{corollary}

\begin{proof}
On every interval
$(a,b) \subset (-T,T)$ with $|\bar x| > \ve$,
$\bar x$ satisfies the Euler-Lagrange equation, which implies the Lagrange-Jacobi identity
\begin{equation}\label{LJ}
\frac{\d^2}{\dt^2}|\bar x(t)|^2 = 2(2-\alpha)V(\bar x(t)).
\end{equation}
Therefore $|\bar x(t)|^2$ is a convex function;
in particular this implies that if there exist $t_1 < t_2$, such that
$|\bar x(t_1)| =  |\bar x(t_2)| =  \ve $, then $|\bar x(t)| =  \ve $ for every
$t \in [t_1,t_2]$, and the corollary follows.
\end{proof}

\begin{lemma}\label{lem:polarEL}
Let $\bar x = \bar r \bar s$ and $t_{*}<t_{**}$ be as above.
Then
\[
\ddot{\bar x}(t) = \nabla_T V(\bar x(t))-\frac{1}{\ve^2}|\dot{\bar x}(t)|^2 \bar x(t),
\quad \text{for every } t \in (t_{*},t_{**})
\]
(here $\nabla_T V$ denotes the tangential part of $\nabla V$, defined in equation \eqref{eq:inpiu}).
\end{lemma}

\begin{proof}
By definition, $\bar x|_{(t_{*},t_{**})}$
minimizes $\action$ with the pointwise constraint$|x(t)|=\ve$.
Applying Lagrange multipliers rule we obtain that
\[
\ddot{\bar x}(t) = \nabla V(\bar x(t)) +\lambda(t)\bar x(t), \quad \text{for every } t \in (t_{*},t_{**}).
\]
We can compute $\lambda$ multiplying by $\bar x(t)$, and recalling that, since
$\bar x(t) \cdot \bar x(t) = \ve^2$, then $\bar x(t) \cdot \dot{\bar x}(t) = 0$ and
$\bar x(t) \cdot \ddot{\bar x}(t) = -|\dot{\bar x}(t)|^2$.
\end{proof}

From the previous discussion it follows that a minimizer $\bar x$ may be not regular
only in $t_{*}$ and $t_{**}$. Our last aim is to study the behavior of $\bar x$ in these points.

\begin{proposition}\label{propo:constraint}
Let $\bar x = \bar r \bar s$ achieve $m$, and $t_{*}$, $t_{**}$ be
defined as in Corollary \ref{cor:convex}.
Then one of the following three situations occurs:
\begin{enumerate}
\item[(a)] $t_{*} < t_{**}$ and $\bar x \in \cont^1(-\bar T, \bar T)$;
\item[(b)] $t_{*} = t_{**}$ and $\bar x \in \cont^1(-\bar T, \bar T)$;
\item[(c)] $t_{*} = t_{**}$ and $\dot{\bar x}(t_*^-) \neq \dot{\bar x}(t_*^+)$; in such a case
$\bar x$ undergoes a radial reflection, that is
\[
\dot{\bar r}(t_*^-)= -\dot{\bar r}(t_*^+) \neq 0
\quad \text{and} \quad
\dot{\bar s}(t_*^-)= \dot{\bar s}(t_*^+).
\]
\end{enumerate}
\end{proposition}

\begin{proof}
We prove the proposition in the case $\ve =1$; the general one follows straightforwardly.
We recall the definition of the \emph{Kelvin transform}:
\[
K \colon   \RR^d \setminus \{0\} \rightarrow \RR^d \setminus \{0\},
\qquad
K(x) := \frac{x}{|x|^2}.
\]
We have that $K$ is a conformal map, $\mathrm{Fix}(K) = \sphere$,
$K^{-1}=K$,
\begin{equation*}\label{eq:K'}
\left[ K'(x)\right]_{ij} = \frac{1}{|x|^2}\left( \delta_{ij} -\frac{2x_ix_j}{|x|^2} \right)
\quad\text{ and }\quad
K'(x)^{T}\,K'(x)
= \frac{1}{|x|^4}I_d.
\end{equation*}
Hence, whenever $x \in \sphere$ and $y \in \RR^d$ we have that
\[
K'(x)y = y-2(x \cdot y)x;
\]
this means that $K'(x)$ is the reflection matrix with respect to the hyperplane orthogonal to $x$.

Let $[-\bar T,\bar T]$ be the definition interval of $\bar x$;
let $\tilde x \in H^1\left(-\bar T,\bar T\right)$, be the path
\[
\tilde x (t) := \begin{cases}
									\bar x(t),  & \text{if } t \in [-\bar T,t_{*}] \\
									K(\bar x(t)),  & \text{if } t \in (t_{*},\bar T],
								\end{cases}.
\]
Using the homogeneity of $V$ we obtain
\[
\begin{split}
\action(\bar x) &
  = \int_{-\bar T}^{t_{*}} \left[\frac12 |\dot{\tilde x}|^2 + V(\tilde x)\right]
   + \int_{t_{*}}^{\bar T}\left[ \frac12 \left( K'(\tilde x)^{T}\,K'(\tilde x)\,\dot{\tilde x} \right)\cdot \dot{\tilde x}  +
    V\left(K(\tilde x)\right)\right] \\
 & = \int_{-\bar T}^{\bar T} \left[ \frac12 \max \left\{ 1,\frac{1}{|\tilde x|^4} \right\}|\dot{\tilde x}|^2
   + \max \left\{ 1,\frac{1}{|\tilde x|^{2\alpha}}\right\} V(\tilde x) \right];
\end{split}
\]
The function $\tilde x$ is then a minimizer for
\[
\action^K(x) = \int_{-\bar T}^{\bar T} \lagr^K(\dot x,x),
\]
\[
{\lagr}^K(\dot x,x):= \frac12 \max \left\{ 1,\frac{1}{|x|^4} \right\}|\dot{x}|^2
+ \max \left\{ 1,\frac{1}{|x|^{2\alpha}}\right\} V(x),
\]
on the set
\(
\left\{x \in H^1(-\bar T,\bar T): x(-\bar T) = \tilde x(-\bar T), x(\bar T) = \tilde x(\bar T) \right\}
\),
without any other constraint.
Since ${\lagr}^K$ is Lipschitz continuous with respect to $x$,
one can prove, by standard arguments in the Calculus of Variations,
that $\tilde x\in\cont^1(-\bar T,\bar T)$ (see for instance \cite{CV1985}).

We now go back to the path $\bar x$.
From Corollary \ref{cor:convex} we deduce that only two different situations can occur:
in the first case $\{t: |\bar x(t)|= |\tilde x(t)|=1\} = \{t_*\}$,
in the second one $\{t: |\bar x(t)|= |\tilde x(t)|=1\} = [t_*,t_{**}]$, with $t_*<t_{**}$.

Let us focus on the first situation; being $\tilde x$ of class $\cont^1$ we have
\[
\dot{\bar x}(t_*^-) = \dot{\tilde x}(t_*^-) = \dot{\tilde x}(t_*)
\]
and
\[
\dot{\bar x}(t_*^+) = K'(\bar x(t_*))\dot{\tilde x}(t_*^+)
                    = K'(\bar x(t_*))\dot{\tilde x}(t_*)
                    = K'(\bar x(t_*))\dot{\bar x}(t_*^-).
\]
As previously remarked, $K'(\bar x(t_*))$ is the reflection matrix with respect
to the hyperplane orthogonal to $\bar x(t_*)$, hence
if $\dot{\bar x}(t_*^-) \cdot {\bar x}(t_*) = 0$ then
$K'(\bar x(t_*))\dot{\bar x}(t_*^-) = \dot{\bar x}(t_*^-)$, $\dot{\bar x}(t_*^+) = \dot{\bar x}(t_*^-)$
and $\bar x \in \cont^1(-\bar T,\bar T)$ (case (b));
otherwise if $\dot{\bar x}(t_*^+) \cdot {\bar x}(t_*) \neq 0$
then $0 < \dot{\bar x}(t_*^+) \cdot {\bar x}(t_*) = -\dot{\bar x}(t_*^-) \cdot {\bar x}(t_*)$.
In this case we can deduce the radial reflection of case (c); indeed,
since $\bar r(t_*)=1$ and $\bar x\cdot \dot {\bar x} = \bar r \dot {\bar r}$, we have
\[
\dot{\bar r}(t_*^+) = \dot{\bar r}(t_*^+) {\bar r}(t_*)=
\dot{\bar x}(t_*^+) \cdot {\bar x}(t_*)= -\dot{\bar x}(t_*^-) \cdot {\bar x}(t_*)
= -\dot{\bar r}(t_*^-),
\]
while the component of the velocity orthogonal to ${\bar x}(t_*)$
is conserved, that is:
\[
\begin{split}
\dot{\bar s}(t_*^+)
&= \dot{\bar x}(t_*^+) - \left( \dot{\bar x}(t_*^+) \cdot {\bar x}(t_*) \right){\bar x}(t_*)\\
&= K'(\bar x(t_*))\left[ \dot{\bar x}(t_*^+)-\left( \dot{\bar x}(t_*^+) \cdot {\bar x}(t_*) \right){\bar x}(t_*) \right]\\
&= K'(\bar x(t_*))\dot{\bar x}(t_*^+) + \left( \dot{\bar x}(t_*^+) \cdot {\bar x}(t_*) \right){\bar x}(t_*)\\
&= \dot{\bar x}(t_*^-) - \left( \dot{\bar x}(t_*^-) \cdot {\bar x}(t_*) \right){\bar x}(t_*)\
= \dot{\bar s}(t_*^-).
\end{split}
\]

Let us now consider the second situation, when the minimizer remains on $\sphere$ for a
nontrivial time interval. Since $\tilde x$ is of class $\cont^1$, both vectors $\dot{\tilde x}(t_*)$ and
$\dot{\tilde x}(t_{**})$ are tangent to $\sphere$ and, still using the properties of $K'$, we have
that $\bar x \in \cont^1(-\bar T,\bar T)$ (case (a)).
\end{proof}

The previous proposition suggests to classify minimizers with respect to the discontinuity of the
quantities $x$ and $\dot x$ on the constraint.

\begin{definition}\label{defi:type_bdd_int}
Let $x =  r  s$ be a constrained Bolza minimizer, and $t_{*}$, $t_{**}$ as above.
Then we can define the following quantities (see Figure \ref{fig:figureJUMPS}):
\[
\frac{\left|x(t_{**}) -  x(t_{*})\right|}{\ve} =  \left|s(t_{**}) -  s(t_{*})\right| =: \pjump(x)
\]
(the normalized position-jump of $x$),
\[
\frac{\dot{ x}(t_{**}^+)\cdot  x(t_{**}) - \dot{ x}(t_*^-)\cdot  x(t_*)}{\ve^{-\alpha/2}\cdot \ve} =
\ve^{\alpha/2}\left[\dot{r}(t_{**}^+) - \dot{r}(t_*^-) \right]=: \vjump(x)
\]
(the normalized  velocity-jump of $x$).\\
Then, according to Proposition \ref{propo:constraint},
\begin{itemize}
\item[(A)] when $ x$ verifies (a) then $t_*<t_{**}$,
$\pjump > 0$, $\vjump =0$ and
we say that $ x$ is \emph{position-jumping};
\item[(B)] when $ x$ verifies (b) then $t_*=t_{**}$,
$\pjump = \vjump =0$ and
we say that $ x$ is \emph{parabolic};
\item[(C)] when $ x$ verifies (c) then $t_*=t_{**}$,
$\pjump = 0$, $\vjump > 0$ and
we say that $ x$ is \emph{velocity-jumping}.
\end{itemize}
\end{definition}
\begin{figure}
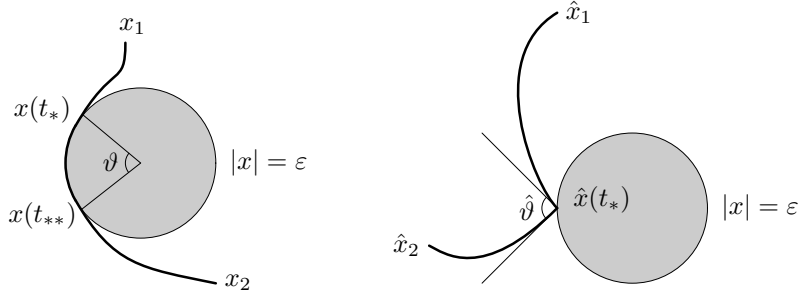

\begin{center}
\begin{mac}
\begin{tabular}{ccc}
\begin{texdraw}
\drawdim cm  \setunitscale 1
%% vincoli
\linewd 0.03
\fcir f:0.8 r:1
\larc r:1 sd:150 ed:210
\linewd 0.01
\move (0 0) \lcir r:1
\move (0 0) \lvec (-0.77 0.64)
\move (0 0) \lvec (-0.788 -0.615)
\move (0 0) \larc r:.2 sd:140 ed:218
%% pjump sopra
\linewd 0.03
\move (-0.866 0.5)\clvec(-0.366 1.366)(-.2 1)(-.2 1.6)
%% pjump sotto
\linewd 0.03
\move (-0.866 -0.5)\clvec(-0.366 -1.366)(0 -1.4)(1 -1.6)
%% scritte
\textref h:C v:C
\htext (-.1 1.8) {$x_1$}
\htext (1.3 -1.6) {$x_2$}
\htext (-1.3 0.7) {$x(t_*)$}
\htext (-1.3 -0.7) {$x(t_{**})$}
\htext (-.4 0) {$\vartheta$}
\htext (1.7 0) {$|x|=\ve$}
\end{texdraw}

& \hspace{1cm} &

\begin{texdraw}
\drawdim cm  \setunitscale 1
%% vincoli
\linewd 0.03
\fcir f:0.8 r:1
\linewd 0.01
\move (-1 0) \lvec (-2 1)
\move (-1 0) \lvec (-2 -1)
\move (-1 0) \larc r:.2 sd:135 ed:225
\move (0 0) \lcir r:1
%% vjump sopra
\linewd 0.03
\move (-1 0)\clvec(-1.3 0.3)(-2 2)(-1 2.6)
%% vjump sotto
\linewd 0.03
\move (-1 0)\clvec(-2 -1)(-2.5 -.6)(-2.7 -.5)
%% scritte
\textref h:C v:C
\htext (-.7 2.6) {$\hat x_1$}
\htext (-3 -.5) {$\hat x_2$}
\htext (-0.4 0.1) {$\hat x(t_*)$}
\htext (-1.4 0) {$\hat \vartheta$}
\htext (1.7 0) {$|x|=\ve$}
\end{texdraw}
\end{tabular}
\end{mac}
\end{center}
\caption{ at left, the trajectory $x$ exhibits a position-jump
$\pjump(x)= 2 \sin(\vartheta/2)$; at right, the trajectory $\hat x$
exhibits a velocity-jump $\vjump(\hat x)= 2 \cos(\hat\vartheta/2)$.\label{fig:figureJUMPS}}
\end{figure}
\begin{remark}\label{rem:hom1}
Since $|x(t^*)|=|x(t^{**})|=\ve$, and, by conservation of energy,
$|\dot x(t)| = \ve^{-\alpha/2}\sqrt{2V(s(t))}$, for any $t \in [t^*,t^{**}]$,
we can rewrite the above quantities in a more readable way.
More precisely,
\[
t_* < t_{**} \quad \implies \quad \pjump(x)=\left|\left[ \frac{x}{|x|}\right]^{t_{**}}_{t_{*}}\right|,
\]
while
\[
t_* = t_{**}
\quad \implies \quad
\frac{\vjump(x)}{\sqrt{2V(s(t_*))}} = \left|\left[ \frac{\dot x}{|\dot x|}\right]^{t^+_{*}}_{t^-_{*}}\right|,
\]
justifying the previous definition.
Furthermore, the (normalized) jumps are invariant by homothetic space-time rescalings.
In fact, it is easy to check that if $x$ achieves $m=m(x_1,x_2,\ep)$, then, for every $R>0$,
\[
z(t):=Rx(R^{-(2+\alpha)/2}t) \quad \text{achieves } m\left(Rx_1,Rx_2,R\ep\right),
\]
and
\[
\pjump(z) = \pjump(x), \qquad  \qquad \vjump(z) = \vjump(x).
\]
\end{remark}

%%%%%%%%%%%%%%%%%%%%%%%%%%%%%%%%%%%%%%%%%%%%%%%
\section{Level Estimates}\label{sec:level_estimates}
%%%%%%%%%%%%%%%%%%%%%%%%%%%%%%%%%%%%%%%%%%%%%%%

In this section we provide a number of estimates on action levels of Bolza minimizers.
The first estimates we state hold for every minimizer, regardless of its jump type.
The main idea consists in comparing their levels with the ones of homothetic solutions
(see Lemmata \ref{lem:omo1} and \ref{lem:omo2} in the appendices, and in particular the definition of the action level $\mathrm{hom}\left(r_1,r_2,\gamma\right)$).
For our future purposes, we make explicit the dependence of the estimates
on the endpoints $x_1$, $x_2$ and on the minimal radius $\ep$. For the reader's convenience, we also recall the definition of $\alpha_*:=(2-\alpha)/2$.
\begin{lemma}\label{lem:est_from_above}
Let $\ve>0$, and $x_1, x_2 \in \RR^d \setminus B_{\ve }(0)$. Then
\[
\begin{split}
 m\left(x_1,x_2,\ep\right)  \leq  \mathrm{hom}\left( \ve, |x_1|, V\left(\frac{x_1}{|x_1|} \right) \right)
                & + \mathrm{hom}\left( \ve, |x_2|, V\left(\frac{x_2}{|x_2|} \right) \right) \\
                & + \frac{\pi}{2} \ve^{\alpha_*} \left|\frac{x_2}{|x_2|} - \frac{x_1}{|x_1|} \right|\sqrt{2V_{\max}}.
\end{split}
\]
\end{lemma}
\begin{figure}
\begin{center}
\begin{mac}
\begin{texdraw}
\drawdim cm  \setunitscale .7
{
%%%%%%%%%%%%%%%%
%% vincoli
%%%%%%%%%%%%%%%%
\linewd 0.03
\fcir f:0.8 r:1
\larc r:1 sd:150 ed:210
\larc r:1 sd:150 ed:210
\larc r:1 sd:-28 ed:30
\linewd 0.01
\lcir r:1
%%%%%%%%%%%%%%%
%% pjump sopra + sotto
%%%%%%%%%%%%%%%
\linewd 0.03
\move (-0.866 0.5)\clvec(-0.366 1.366)(1.8 3.1)(7 4)
\move (-0.866 -0.5)\clvec(-0.366 -1.366)(1.8 -2.8)(6 -3)
%%%%%%%%%%%%%%%%%%
%% omografiche tratteggiate
%%%%%%%%%%%%%%%%%%
\linewd 0.01
\setgray 0.5 \lpatt (0.2 0.2)
\move (7 4) \lvec(0 0)
\move (6 -3) \lvec(0 0)
\lpatt()
\setgray 0
%%%%%%%%%%%%%%%%%%
%% altri cammini
%%%%%%%%%%%%%%%%%%
\linewd 0.03
\move (7 4)\clvec(1.8 2.1)(1 1)(.86 .49)
\move (6 -3)\clvec(1.8 -1.3)(1 -1)(.89 -.45)
%%%%%%%%%%%%%%%%
%% frecce
%%%%%%%%%%%%%%%%
\linewd 0.03
\move (3.5 3.16) \arrowheadtype t:V  \arrowheadsize l:0.3 w:0.2 \avec (3.45 3.15)
\move (1 0.05) \arrowheadtype t:V  \arrowheadsize l:0.24 w:0.16 \avec (1 0)
\move (1.5 1.32) \arrowheadtype t:V  \arrowheadsize l:0.3 w:0.2 \avec (1.4 1.23)
\move (1.5 -1.07) \arrowheadtype t:V  \arrowheadsize l:0.3 w:0.2 \avec (1.6 -1.12)
%%%%%%%%%%%%%%%%%%
%% scritte
%%%%%%%%%%%%%%%%%%
\textref h:C v:C
\htext (3.5 3.5) {$x$}
\htext (7.6 4.2) {$x_1$}
\htext (6.2 -3.2) {$x_2$}
\htext (1.2 0) {$\sigma$}
\htext (1.3 1.6) {$\gamma_1$}
\htext (1.5 -1.4) {$\gamma_2$}
}
\end{texdraw}
\end{mac}
\end{center}
\caption{ \label{fig:lemma4.1} the paths $\gamma_1$, $\gamma_2$, $\sigma$ and $x$ achieve
respectively $m\left(x_1,\ve x_1/|x_1|,\ve  \right)$,
$m\left(\ve x_2/|x_2|,x_2,\ve  \right)$,
$m\left(\ve x_1/|x_1|,\ve x_2/|x_2|,\ve  \right)$ and $m\left(x_1,x_2,\ve\right)$.}
\end{figure}
\begin{proof}
To prove the required estimate we observe that, using the notation in equation \eqref{eq:minA}, there holds
\[
m\left(x_1,x_2,\ep\right)\leq m\left(x_1,\ep \frac{x_1}{|x_1|},\ep\right)
+ m\left(\ep \frac{x_1}{|x_1|},\ep \frac{x_2}{|x_2|},\ep\right)
+ m\left(\ep \frac{x_2}{|x_2|},x_2,\ep\right)
\]
(indeed, by juxtaposing paths in $\Gamma(x_1,\ep x_1/|x_1|,\ep)$,
$\Gamma(\ep x_1/|x_1|,\ep x_2/|x_2|,\ep)$, and $\Gamma(\ep x_2/|x_2|,x_2,\ep)$,
we obtain a path in $\Gamma(x_1,x_2,\ep)$, see Figure \ref{fig:lemma4.1}). But then, on one hand, Lemma \ref{lem:omo2} yields, for $i=1,2$,
\[
{m}\left(\ep \frac{x_i}{|x_i|},x_i,\ep\right) \leq \mathrm{hom}
\left( \ve, |x_i|,  V\left(\frac{x_i}{|x_i|} \right)\right).
\]
On the other hand, to estimate $m(\ep x_1/|x_1|,\ep x_2/|x_2|,\ep)$, let us assume that $\chi(t) = \ep \sigma(t)$ is \emph{any} trajectory joining the considered endpoints on
(say) $[-T,T]$, with the further property to have zero energy (as we observed, any path can be parameterized in this way, recall Lemma \ref{lem:maup}), that is
\[
\frac12 \ep^2|\dot \sigma(t)|^2 = \frac{V(\sigma(t))}{\ep^{\alpha}},
\quad \text{so that} \quad
|\dot \sigma(t)| = \ep^{-(2+\alpha)/2}\sqrt{2V(\sigma(t))}.
\]
Defining the arc-length parameter $\vartheta (t):= \int_{-T}^t |\dot \sigma(\tau)|
\dtau$, we have that
%\begin{multline*}%\label{eq:azione_sul_cerchio}
\[
\action(\chi) = \int_{-T}^T \ep^2|\dot \sigma(t)|^2 \dt
%=\int_0^{\vartheta(T)} \ep^2 |\dot \sigma(t(\vartheta))| \d\vartheta \\
= \ep^{\alpha_*} \int_0^{\vartheta(T)} \sqrt{2V(\sigma(t(\vartheta)))}\d\vartheta.
\]
%\end{multline*}
Choosing $\sigma$ to be the geodesic on the sphere between $x_1/|x_1|$ and $x_2/|x_2|$,
the lemma follows (indeed the maximal value assumed by the ratio between
the length of an arc and the one of the correspondent chord is $\pi/2$).
\end{proof}
Of course, while the estimate from above holds just for minimizers,
the one from below can be extended to every path satisfying the constraints. Such
an estimate can be improved, once one knows that the considered path crosses a zone
where the angular part of the potential is really greater than $V_{\min}$.
\begin{lemma}\label{lem:est_from_below}
Let $\bar x \in \Gamma(x_1,x_2,\ep)$, $\bar x=\bar r\bar s$, be a path
such that
\begin{itemize}
\item $x_1=\bar x(T_1)$, $x_2=\bar x(T_2)$;
\item for some $a\leq b$, $|\bar x(t)|=\ep$ on $[a,b]$.
\end{itemize}
Finally, for any $[t_1,t_2] \subset [T_1,T_2]\setminus(a,b)$, let us define the quantities
\[
\gamma:= \min_{t\in[t_1,t_2]}\left( V(\bar s(t))-V_{\min} \right)
\qquad \text{and} \qquad \bar r_{\min}:= \min_{t\in[t_1,t_2]} \bar r(t).
\]
Then
\begin{multline*}
 \action(\bar x) \geq  \mathrm{hom}\left( \ep, |x_1|, V_{\min} \right)
                 + \mathrm{hom}\left( \ep, |x_2|, V_{\min} \right)\\
                 + \sqrt{2V_{\min}}\, \ep^{\alpha_*}|\bar s(b)-\bar s(a)|
                 + \sqrt{2\gamma}\, \bar r_{\min}^{\alpha_*}|\bar s(t_2)-\bar s(t_1)|.
\end{multline*}
\end{lemma}
\begin{proof}
We have that, for every $[t_1,t_2] \subset [T_1,T_2]\setminus(a,b)$,
\begin{multline*}
%\[
%\begin{split}
\action(\bar x)  \geq
\int_{T_1}^{a} \left[\frac12 \dot{\bar r}^2 +
\frac{V_{\min}}{\bar r^\alpha}\right] +
\int_{b}^{T_2} \left[\frac12 \dot{\bar r}^2 +
\frac{V_{\min}}{\bar r^\alpha}\right] \\+
\int_{a}^{b} \left[\frac12 \bar \ep^2 |\dot{\bar s}|^2 +
\frac{V_{\min}}{\bar \ep^\alpha}\right] +
\int_{t_1}^{t_2} \left[\frac12 \bar r^2 |\dot{\bar s}|^2 +
\frac{V(\bar s)-V_{\min}}{\bar r^\alpha} \right] . %\\
%& \geq \mathrm{hom}\left(\ep, |x_1|, V_{\min} \right)
%     + \mathrm{hom}\left(\ep, |x_2|, V_{\min} \right) +
%\int_{t_1}^{t_2} \left[\frac12 \bar r^2 |\dot{\bar s}|^2 +
%\frac{V(\bar s)-V_{\min}}{\bar r^\alpha} \right] \dt
%\end{split}
%\]
\end{multline*}
The first two terms are controlled by the definition of $\mathrm{hom}$ (Lemma \ref{lem:omo1}). The remaining ones are of the same type, and they can be estimated analogously. We give the details for the last one, being the other slightly easier.
We want to estimate from below the action
\[
\action_\gamma(\bar x)=\int_{t_1}^{t_2} \left[\frac12 \bar r^2 |\dot{\bar s}|^2 +
\frac{\gamma}{\bar r^\alpha} \right].
\]
Exactly as in the proof of the previous lemma, let $\chi(t) = \rho(t) \sigma(t)$ be \emph{any} trajectory joining the considered endpoints on $[-T,T]$, with
%
%
%We want to estimate from below the quantity
%\[
%m_*:=\inf
% \left\{ \int_{-T}^{T} \left[\frac12 r^2 |\dot s|^2 + \frac{\gamma}{r^\alpha} \right] \dt
%:\,
%   \begin{array}{l}
%      T>0, \, x=rs \in H^1(-T,T), \smallskip\\
%      x(-T)=\bar x(t_1),\, x(T)=\bar x(t_2), \smallskip\\
%      r \geq \bar r_{\min}
%   \end{array}
%\right\}.
%\]
%Reasoning as in Lemmata \ref{lem:legameAJ}, \ref{lem:Boltza_exists}, it is easy to prove that $m_*$ is achieved by some $\chi=\rho\sigma$, that, by Corollary \ref{lem:generale}, satisfies
\[
\frac12 \rho^2(t) |\dot \sigma(t)|^2 = \frac{\gamma}{\rho^\alpha(t)},
\quad \text{so that} \quad
|\dot \sigma(t)| = \rho^{-(2+\alpha)/2}(t)\sqrt{2\gamma}.
\]
Defining the arc-length parameter $\vartheta (t):=
\int_{-T}^t |\dot \sigma(\tau)| \dtau$,
we have that
\begin{multline*}
\action_\gamma(\chi) = \int_{-T}^T \rho^2(t)|\dot \sigma(t)|^2 \dt =
 \int_0^{\vartheta(T)} \rho^{\alpha_*}(t(\vartheta))\sqrt{2\gamma}\d\vartheta\\
\geq \bar r_{\min}^{\alpha_*}|\sigma(T)-\sigma(-T)|\sqrt{2\gamma}. \qedhere
\end{multline*}
\end{proof}
Now on we want to sharpen the previous level estimates making use of the definition
of $\pjump$ and $\vjump$.
To start with we observe that, on intervals where the Euler-Lagrange equation holds,
the corresponding action level can be rewritten in terms of the minimizer endpoints.

\begin{lemma}\label{lem:stima_azione}
Let $\bar x = \bar r\bar s$ be a Bolza minimizer,
and $(a,b)\subset \{t:\,\dot {\bar r}(t)<0\}\cup\{t:\,\dot {\bar r}(t)>0\}$.
Then
\[
\action([a,b];\bar x)=\frac{2}{2-\alpha}\left[\bar x(t)\cdot\dot {\bar x}(t)\right]_a^b
                = \frac{1}{\alpha_*}\left[\bar r(t)\dot {\bar r}(t)\right]_a^b.
\]
Moreover, if $\bar r(b)>\bar r(a)$, then
\begin{multline*}
\frac{\sqrt{2V_{\min}}}{\alpha_*}\,\bar r^{\alpha_*}(b) - \frac{1}{\alpha_*}\,\bar r(a)\dot {\bar r}(a^+)
          -\frac{\sqrt{2V_{\min}}}{\alpha_*}\,\frac{\bar r^{2\alpha_*}(a)}{\bar r^{\alpha_*}(b)}\\
\leq
\action([a,b];\bar x)
\leq
\frac{\sqrt{2V(\bar s(b))}}{\alpha_*}\,\bar r^{\alpha_*}(b) - \frac{1}{\alpha_*}\,\bar r(a)\dot {\bar r}(a^+)
\end{multline*}
(if $\bar r(a)>\bar r(b)$, an analogous estimate holds).
\end{lemma}
\begin{proof}
By Corollary \ref{cor:convex}, on $(a,b)$ the Euler-Lagrange equation holds.
Multiplying by $\bar x$, integrating by parts, and using homogeneity, we have
\[
\left[\bar x(t)\cdot \dot {\bar x}(t)\right]_a^b - \int_a^b |\dot {\bar x}|^2 =
-\alpha\int_a^b V(\bar x)).
\]
As a consequence, conservation of energy yields
\[
\left[\bar x(t)\cdot\dot {\bar x}(t)\right]_a^b =
(2-\alpha)\int_a^b \frac12|\dot {\bar x}|^2 =
(2-\alpha)\int_a^b V(\bar x)
= \alpha_*\action([a,b];\bar x).
\]

Now let us assume $\bar r(b)>\bar r(a)$, so that $\dot {\bar r}>0$ on $(a,b]$.
To prove  the estimate from above, we use the fact that,
by conservation of energy,
\[
\frac12 \dot {\bar r}^2(t) \leq \frac{V(\bar s(t))}{\bar r^\alpha(t)},
\quad
\text{ which implies }
\quad
\bar r(b) \dot {\bar r}(b) \leq \sqrt{2V(\bar s(b))}\,\bar r^{\alpha_*}(b).
\]
In order to obtain the estimate from below, we define the auxiliary function
\[
\vp(t) = \frac12 \bar r^2(t){\dot {\bar r}}^2(t) - V_{\min}\bar r^{2\alpha_*}(t).
\]
By direct computation, using \eqref{eq:ELrs}, we have
\[
\dot \vp(t) = 2\alpha_* \bar r^{1-\alpha}(t) \left[ V(\bar s(t))-V_{\min}\right]{\dot {\bar r}}(t) \geq 0,
\]
thus $\vp(t)$ is an increasing function.
In particular $\vp(b) \geq \vp(a)$, and we obtain the following chain of inequalities
\[
\begin{split}
\bar r(b)\dot {\bar r}(b) & \geq \sqrt{\bar r^2(b)\dot {\bar r}^2(b) - \bar r^2(a)\dot {\bar r}^2(a)}%\\&
              \geq \sqrt{2V_{\min}(\bar r^{2\alpha_*}(b)-\bar r^{2\alpha_*}(a))}\\&
              \geq \sqrt{2V_{\min}}\left[\bar r^{\alpha_*}(b)
                            -\bar r^{2\alpha_*}(a)\bar r^{-\alpha_*}(b)\right]
\end{split}
\]
(we used the elementary inequality $\sqrt{A^2-B^2}\geq A - (B^2/A)$).
Subtracting $\bar r(a)\dot {\bar r}(a)$ we obtain the desired estimate.
\end{proof}

\begin{lemma}\label{lem:stima_inf}
Let $\ve>0$, and $x_1$, $x_2$ be such that $|x_1|=|x_2|=R>\ve$.
If $\bar x = \bar r\bar s$ achieves $m(x_1,x_2,\ve)$ and $\vjump(\bar x)=0$,
then
\begin{multline*}
\frac{2\sqrt{2V_{\min}}}{\alpha_*}\,\left[R^{\alpha_*}
-\frac{\ve^{2\alpha_*}}{R^{\alpha_*}}\right] +\sqrt{2V_{\min}}\,\pjump(\bar x)\,\ve^{\alpha_*}
\leq \action\left(\bar x\right) \\
\leq \frac{\sqrt{2V(s_1)} + \sqrt{2V(s_2)}}{\alpha_*}\, R^{\alpha_*}+
\frac{\pi\sqrt{2V_{\max}}}{2} \pjump(\bar x)\, \ve^{\alpha_*}.
\end{multline*}
\end{lemma}
\begin{proof}
To fix the ideas let $[T_1,T_2]$ be the definition interval of $\bar x$ and
$t_{*}$, $t_{**}$ be defined as in Definition \ref{defi:type_bdd_int}, so that
\[
\action\left([T_1,T_2];\bar x\right) = \action\left([T_1,t_*];\bar x\right)+
m\left(\bar x(t_{*}),\bar x(t_{**}),\ep\right)
+\action\left([t_{**},T_2];\bar x\right).
\]
We recall that, by assumption, $\bar r(t_*)=\bar r(t_{**})=\ep$
and (since $\vjump(\bar x)=0$) $\dot {\bar r}(t_*)=\dot {\bar r}(t_{**})=0$.
We can easily estimate the first and the last term applying twice Lemma \ref{lem:stima_azione}.
Finally, Lemmas \ref{lem:est_from_above} and \ref{lem:est_from_below}
and the definition of $\pjump(\bar x)$ imply
\[
\sqrt{2V_{\min}}\, \ve^{\alpha_*}\pjump(\bar x)
\leq
m\left(\bar x(t_{*}),\bar x(t_{**}),\ep\right)
\leq
\frac{\pi}{2}\sqrt{2V_{\max}}\, \ve^{\alpha_*} \pjump(\bar x).\qedhere
\]
\end{proof}
\begin{lemma}\label{lem:stima_sup}
Let $\ve>0$, and $x_1$, $x_2$ be such that $|x_1|=|x_2|=R>\ve$.
If $\bar x = \bar r\bar s$ achieves $m(x_1,x_2,\ve)$ and $\pjump(\bar x)=0$,
then
\begin{multline*}
\frac{2\sqrt{2V_{\min}}}{\alpha_*}\,\left[R^{\alpha_*}
-\frac{\ve^{2\alpha_*}}{R^{\alpha_*}}\right] -\frac{1}{\alpha_*}
\vjump(\bar x)\, \ve^{\alpha_*} \\
\leq \action\left(\bar x\right)
\leq \frac{\sqrt{2V(s_1)} + \sqrt{2V(s_2)}}{\alpha_*}\, R^{\alpha_*}
-\frac{1}{\alpha_*} \vjump(\bar x)\, \ve^{\alpha_*}.
\end{multline*}
\end{lemma}
\begin{proof}
Let again $T_1$, $T_2$, $t_{*}$ and $t_{**}$ be defined as at the beginning of the previous proof.
By Proposition \ref{propo:constraint} we have that $t_* = t_{**} = 0$,
obtaining that
\[
\action\left([T_1,T_2];\bar x\right) = \action\left([T_1,0^-];\bar x\right)+
\action\left([0^+,T_2];\bar x\right).
\]
Therefore we can conclude applying again Lemma \ref{lem:stima_azione} and recalling that, by Definition \ref{defi:type_bdd_int}, there holds
\[
-\bar r(0)\dot {\bar r}(0^-) = \bar r(0)\dot {\bar r}(0^+) = \frac12 \ve^{\alpha_*} \vjump(\bar x).\qedhere
\]
\end{proof}
%
%====================================
\section{Morse Minimizers}\label{sec:Morse}
%====================================

Throughout this section the potential $V \in \Poh$ and $\ve>0$
are fixed (in fact, the role of $\ve$ can be ruled out by scaling,
see Remark \ref{rem:hom2}).
\begin{definition}\label{defi:constr_Morse_min}
We say that $x\in H^1_{\mathrm{loc}}(\RR)$ is an \emph{$\ve$-constrained Morse minimizer} if
\begin{itemize}
 \item $\min_t |x(t)|=\ve$;
 \item $|x(t)|\to+\infty$ and $x(t)/|x(t)|\to\xi^\pm$, as $t\to\pm\infty$;
 \item for every $a<b$, $a'<b'$, and $z\in H^1(a',b')$, there holds
\begin{multline*}
z(a')=x(a),\ z(b')=x(b),\,\min_{[a',b']}|z|=\min_{[a,b]}|x| \\
\implies\quad \action([a,b];x)\leq\action([a',b'];z).
\end{multline*}
\end{itemize}
We denote with
$\morse$ the set of $\ve$-constrained Morse minimizers.
\end{definition}
Actually, since zero-energy trajectories defined on unbounded intervals
can not be bounded, the condition $|x(\pm\infty)|=+\infty$ is unnecessary (see \cite{Chenciner1998}
and references therein).

The main idea in the proof of the existence of $\ve$-constrained Morse minimizers
is to argue by approximation, solving the Bolza problem \eqref{eq:minA} with
$x_1=R\xi^-$ and $x_2 = R\xi^+$ and then letting $R \to +\infty$.
Such a procedure provides a trajectory in $H^1_{\mathrm{loc}}(\RR)$
(Appendix \ref{app:stab}), that turns out to be asymptotic to some central configurations
(Appendix \ref{app:stime}). Thus, the main thing to prove is that such configurations are
indeed $\xi^{\pm}$.
\begin{lemma}\label{lem:Morse_not_empty}
$\morse$ is not empty.
\end{lemma}
\begin{proof}
As we mentioned above, we can construct an element of $\morse$
as limit of suitable Bolza minimizers. To this aim, let
$n \in \NN$ (large) and, following Section \ref{sec:Bolza}, let $x_n(t)$
be a solution of the minimization problem \eqref{eq:minA}, with endpoints $x_1= n\xi^-$ and $x_2= n\xi^+$.
Using Proposition \ref{propo:constraint}, we can associate with each $x_n$ the times $t_{*,n}\leq t_{**,n}$, in which it interacts with the constraint.

First of all, let us observe that
\[
t_{**,n}-t_{*,n} \leq C,
\]
independent of $n$; indeed by Lemma \ref{lem:est_from_above}
\begin{multline*}
\frac{\pi}{2} \ve^{\alpha_*} |\xi_+-\xi_-|\sqrt{2V_{\max}}
\geq \action\left([t_{*,n},t_{**,n}];x_n\right) \\
= 2\int_{t_{*,n}}^{t_{**,n}} \frac{V(s_n)}{r_n^{\alpha}}  \geq 2\frac{V_{\min}}{\ve^{\alpha}}(t_{**,n}-t_{*,n}).
\end{multline*}
Up to a time translation, we can assume that $t_{*,n}\leq 0 \leq  t_{**,n}$, in such a way that
$x_n$ is defined, say, on $[T_{1,n},T_{2,n}]$.
Using Lemma \ref{lem:est_from_below} we have that
\[
T_{1,n}\to -\infty \quad \text{ and } \quad T_{2,n}\to +\infty:
\]
indeed, for instance,
\begin{multline*}
C n^{\alpha_*} \leq \mathrm{hom}(\ve,n,V_{\min}) \leq \action\left([t_{**,n},T_{2,n}];x_n\right) \\
= 2\int_{t_{**,n}}^{T_{2,n}} \frac{V(s_n)}{r_n^{\alpha}}  \leq 2\frac{V_{\max}}{\ve^{\alpha}}(T_{2,n}-t_{**,n}).
\end{multline*}
Since each $x_n$ satisfies a differential equation separately on $(T_{1,n},t_{*,n})$,
$(t_{*,n},t_{**,n})$ and $(t_{**,n},T_{2,n})$ (see Corollary \ref{cor:convex} and Lemma \ref{lem:polarEL}),
and $x_n$, $\dot x_n$ are uniformly bounded on $[t_{*,n}-1,t_{**,n}+1]$,
we can apply Theorem \ref{teo:stab} (three times). We obtain
that, up to a subsequence,
\[
x_n \to x \text{ in } H^1_{\mathrm{loc}}(\RR),
\]
pointwise on $\RR$, uniformly on every compact interval, and $\cont^2$ outside of
two suitable times $t_{*}$ and $t_{**}$, where suitable Euler-Lagrange equations hold;
moreover also conservation of (zero-)energy is satisfied.

We claim that $x \in \morse$. Actually, the first property of Definition \ref{defi:constr_Morse_min}
is trivially satisfied, while the third one can be easily deduced by contradiction, using the minimality of
$x_n$. Of course $|x(t)| \to +\infty$, therefore we are left to prove that its
limiting configurations are exactly $\xi^{\pm}$.

Let us assume by contradiction that, for instance, $s(t)=x(t)/|x(t)| \not \to \xi^+$, as $t \to +\infty$.
Nevertheless, since $x$ satisfies the assumptions of Theorem \ref{teo:stime},
we have that $\nabla_T V(s(t)) \to 0$. Since $\xi^+$ is an isolated point in $\{s : \, \nabla_T V(s)= 0\}$,
we infer that $|s(t)-\xi^+| \geq \delta >0$ for a suitable $\delta$ and large $t$.
Since $s_n\stackrel{n}{\rightarrow} s$ uniformly on compact sets we deduce the existence of
$t_{1,n},t_{2,n} \to +\infty$ such that
$t_{1,n} < t_{2,n} < T_{2,n}$ and, for some suitable $\eta_1,\eta_2 >0$
\[
\min_{t \in [t_{1,n},t_{2,n}]}\left| V(s_n(t))-V_{\min} \right| \geq \eta_1
\qquad \text{and} \qquad
\left|s_n(t_{1,n}) - s_n(t_{2,n}) \right|\geq \eta_2,
\]
for every $n$.
Applying Lemma \ref{lem:est_from_below} we have that
\begin{multline*}
\action([T_{1,n},T_{2,n}];x_n)
\geq 2\,\mathrm{hom}\left(\ep, n, V_{\min} \right) \\
\qquad + \sqrt{2\min_{t\in[t_{1,n},t_{2,n}]}\left( V(s_n(t))-V_{\min} \right)}
           \left[\min_{t\in[t_{1,n},t_{2,n}]} r_n^{\alpha_*} (t)\right] |s_n(t_{2,n})-s_n(t_{1,n})|
                 \\
\geq 2\,\mathrm{hom}\left(\ep, n, V_{\min} \right) + \sqrt{2\eta_1} \, \eta_2 \, r_n^{\alpha_*} (t_{1,n}).
\end{multline*}
On the other hand, according to Lemma \ref{lem:est_from_above} we obtain
\[
\action([T_{1,n},T_{2,n}];x_n) \leq
2\mathrm{hom}\left( \ve, n, V_{\min} \right)
+ \frac{\pi}{2} \ve^{\alpha_*} \left|\xi_+ - \xi_-\right| \sqrt{2V_{\max}}.
\]
Since, as $t_{1,n} \to +\infty$, $r_n(t_{1,n})\to +\infty$, we obtain a contradiction.
\end{proof}
\begin{remark}\label{rem:freeconv}
Reasoning as above, one may try to obtain a free (parabolic) Morse minimizer as limit of a sequence $x_n$ of
free Bolza ones. In this direction two problems arise: on one hand, it may happen that the sequence
escapes from every bounded domain; on the other hand, it may converge to collision.
Actually this is the main reason for which we decided to introduce the constraint. Nonetheless,
if one may ensure that, for every $n$, $0<C_1 \leq \min |x_n| \leq C_2 <+\infty$ then the
previous procedure would lead to a free Morse minimizer.
\end{remark}
Since by definition any restriction of a Morse minimizer is indeed a Bolza one
(with the appropriate constraint), we have that also Morse minimizers can be classified
according to their jumps, exactly as in Definition \ref{defi:type_bdd_int}.

\begin{lemma}\label{lem:descrizMorse}
Let $x=rs \in\morse$. Then (up to a time translation, that now on will be left out)
there exist $t_* \leq 0 \leq t_{**}$ such that:
\begin{enumerate}
\item $r(t) = \ve$ if and only if $t \in [t_{*},t_{**}]$,
$\dot{r}(t)< 0$ (resp. $>0$) if and only if $t <t_*$ (resp. $t > t_{**}$);
\item $\ddot x(t) = \nabla V(x(t))$, for every $t \not\in [t_{*},t_{**}]$;
\item letting $\pjump(x)$ and $\vjump(x)$ be defined as in Definition \ref{defi:type_bdd_int},
then both of them are non-negative and at least one vanishes.
\end{enumerate}
Moreover:
\begin{enumerate}
\item[4.] $\dfrac12 |\dot x(t)|^2 = V(x(t))$, for every $t \in \RR$.
\end{enumerate}
\end{lemma}

\begin{proof}
If $T_1\ll0$ and $T_2\gg0$ then
$x|_{[T_1,T_2]}$ achieves $m(x(T_1),x(T_2),\eps)$. The results follow from
Corollary \ref{cor:convex}, Proposition \ref{propo:constraint} and Corollary \ref{coro:energy}.
\end{proof}

In general, for any fixed $\ve$ (and potential $V$),
there is no reason to expect uniqueness for the Morse minimizers.
Nevertheless, it is possible to show that, with respect to the jump classification,
they are all of the same type. In order to do that we need to sharpen the asymptotic estimates
contained in Appendix \ref{app:stime}, exploiting the fact that $\xi^{\pm}$ are non degenerate minima
of $V|_{\sphere}$.
\begin{lemma}\label{lem:asintotiche}
Let $x = rs\in\morse$.
Then
\[
\lim_{t \to \pm\infty} [r(t)]^{\alpha_*}|s(t)-\xi^{\pm}| = 0.
\]
\end{lemma}
\begin{proof}
We prove the lemma in the case $t\to+\infty$. Given $\gamma >0$ (to be chosen later),
we define the function
\[
u(t) := t^{-\gamma}\cdot [r(t)]^{-\alpha_*}.
\]
We obtain that $u > 0$, $u \to 0$ as $t \to +\infty$, and
\[
\begin{split}
\frac{\d}{\d t}(r^2\dot u)
         &= \frac{\d}{\d t} \left( -\gamma t^{-\gamma-1}\cdot r^{2-\alpha_*}
                        - \alpha_* t^{-\gamma}\cdot r^{1-\alpha_*}\dot r \right) \\
         &= \gamma(\gamma+1) t^{-\gamma-2}\cdot r^{2-\alpha_*} -  2(1-\alpha_*)\gamma t^{-\gamma-1}\cdot
               r^{1-\alpha_*}\dot r \\
         & \qquad - \alpha_* t^{-\gamma}\cdot r^{-\alpha_*}\left(({1-\alpha_*})\dot r^2
               + r\ddot r\right)\\
         &=  \gamma(\gamma+1) t^{-\gamma-2}\cdot r^{2-\alpha_*} -  2(1-\alpha_*)\gamma t^{-\gamma-1}\cdot
               r^{1-\alpha_*}\dot r \\
         & \qquad - \alpha_* t^{-\gamma}\cdot r^{-\alpha_*}\left(\alpha_* |\dot s|^2r^2\right)\\
         &=  r^2 u \left[\gamma(\gamma+1) t^{-2} -  \alpha \gamma t^{-1}\cdot
               r^{-1}\dot r - \alpha_*^2 |\dot s|^2 \right],
\end{split}
\]
so that, recalling that $\dot r(t)>0$ for $t$ large,
\[
\begin{split}
\ddot u + 2\dfrac{\dot r}{r}\dot u \leq \frac{\gamma(\gamma+1)}{t^2}\, u.
\end{split}
\]
Analogously, letting
\[
v(t) := |s(t)-\xi^+|,
\]
we have $v > 0$, $v \to 0$ as $t  \to +\infty$, and
\[
\dot v (t ) = -\frac{\xi^+ \cdot \dot s }{|s-\xi^+|}, \qquad
\ddot v (t ) = -\frac{\xi^+ \cdot \ddot s }{|s-\xi^+|} + \frac{|\xi^+ \cdot \dot s |^2}{|s-\xi^+|^3}
\]
(recall that $s\cdot \dot s\equiv0$), implying
\[
\begin{split}
\ddot v  + 2\dfrac{\dot r  }{r }\dot v  & = \frac{|\xi^+ \cdot \dot s |^2}{|s-\xi^+|^3}- \frac{\xi^+}{|s-\xi^+|} \cdot
    \left( \ddot s  + 2\dfrac{\dot r  }{r }\dot s   \right) \smallskip \\
    & \geq - \frac{\xi^+}{|s-\xi^+|} \cdot \left( \dfrac{\nabla_T V(s)}{r^{2+\alpha}} -  |\dot s |^2 s  \right) \smallskip\\
    & \geq \frac{\left(\nabla_T V(s) - \nabla_T V(\xi^+)\right)\cdot (s-\xi^+)}{|s-\xi^+|\,\,r^{2+\alpha}}
\end{split}
\]
for large $t $, since $s \cdot \xi^+ \to 1$ as $t  \to +\infty$.
Let us observe that, since $V\in \Seh$, we have that whenever $s\in\sphere$ is sufficiently close to $\xi^\pm$
there holds
\begin{equation*}%\label{eq:non_deg_grad}
\left(\nabla_T V(s)-\nabla_T V(\xi^\pm)\right)\cdot (s-\xi^\pm) \geq
2\mu |s-\xi^\pm|^2.
\end{equation*}
Taking into account Corollary \ref{coro:lim_r},
we have that, for $t$ sufficiently large,
\[
\frac{\left(\nabla_T V(s) - \nabla_T V(\xi^+)\right)\cdot (s-\xi^+)}{|s-\xi^+|}
\cdot\frac{1}{r^{2+\alpha}} \geq 2\mu |s-\xi^+| \cdot \frac{1}{2\Upsilon_+^2 t^2},
\]
in such a way that
\[
\ddot v  + 2\dfrac{\dot r  }{r }\dot v  \geq \frac{\mu}{\Upsilon_+^2 t^2} \,v.
\]
Assuming that the previous inequalities hold for, say, $t\geq\tau$, we
infer that the function
\[
w(t ) := u(t ) - \frac{u(\tau)}{v(\tau)}v(t )
\]
satisfies
\[
\begin{cases}
\displaystyle\ddot w + 2\dfrac{\dot r  }{r }\dot w \leq \frac{\gamma(\gamma+1)}{t^2}\, w
+ \frac{u(\tau)}{v(\tau)}\left(\gamma(\gamma+1)-\frac{\mu}{\Upsilon_+^2}\right)\frac{1}{t^2}\,v, \smallskip\\
w(\tau) = \lim_{t  \to +\infty} w(t ) =0.
\end{cases}
\]
Let us now choose $\gamma$ sufficiently small, in such a way that
\[
\gamma(\gamma+1)-\frac{\mu}{\Upsilon_+^2}<0,
\]
and let us assume by contradiction that $w$ is not everywhere positive. As a
consequence there exists $\bar t >\tau$ such that
$w(\bar t )\leq 0$, $\dot w(\bar t )= 0$, $\ddot w(\bar t )\geq 0$.
Substituting in the equation for $w$ this yields a contradiction, therefore, for $t\geq\tau$,
\[
u(t ) - \frac{u(\tau)}{v(\tau)}v(t )>0,\quad\text{ that is, }\quad
[r(t)]^{\alpha_*}|s(t)-\xi^+|\leq Ct^{-\gamma}. \qedhere
\]
\end{proof}
The previous estimates provide a very strong control on the
action of the tails of Morse minimizers.
\begin{lemma}\label{lem:unico_tipo}
Let $x,\,\hat x \in \morse$. Then
\[
\lim_{R \to +\infty} \left|\action(x|_{\{|x|\leq R\}}) - \action(\hat x|_{\{|\hat x|\leq R\}})\right| = 0.
\]
\end{lemma}
\begin{proof}
As a notation, we write $x=rs$, $\hat x= \hat r \hat s$ and, for $R$ large,
\[
\left\{ t : \, |x(t)|\leq R\right\} = [T_1,T_2], \quad
\left\{ t : \, |\hat x(t)|\leq R\right\} = [\hat T_1,\hat T_2].
\]
Let us observe that the corresponding restrictions are Bolza minimizers, with the suitable endpoints.
Recalling the definition of $m=m(x_1,x_2,\ve)$
we have that
\begin{multline}\label{eq:unico_tipo_1}
\action \left([T_1,T_2]; x \right) \leq \action \left([\hat T_1,\hat T_2]; \hat x \right) \\
+ m\left(x(T_1),\hat x(\hat T_1),R  \right) + m\left(\hat x(\hat T_2),x(T_2),R  \right)
\end{multline}
(see Figure \ref{fig:3}).
\begin{figure}
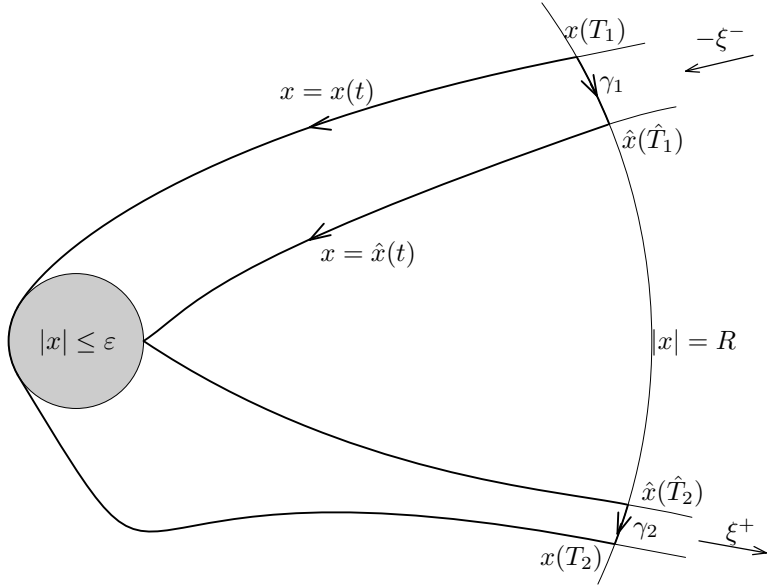

\begin{center}
\begin{mac}
\begin{texdraw}
\drawdim cm  \setunitscale .9
{
%%%%%%%%%%%%%%%%
%% vincoli
%%%%%%%%%%%%%%%%
\linewd 0.03
\fcir f:0.8 r:1
\larc r:1 sd:150 ed:210
\linewd 0.01
\lcir r:1
\larc r:8.5 sd:-25 ed:36
%%%%%%%%%%%%%%%%%%
%% configurazioni
%%%%%%%%%%%%%%%%%%
\linewd 0.02 \setgray 0 \lpatt()
\move (10 4.225) \arrowheadtype t:V \arrowheadsize l:0.2 w:0.12 \avec (9 3.99)
\move (9.2 -2.95) \arrowheadtype t:V \arrowheadsize l:0.2 w:0.12 \avec (10.2 -3.13)
%%%%%%%%%%%%%%%
%% pjump sopra
%%%%%%%%%%%%%%%
\linewd 0.03
\move (-0.866 0.5)\clvec(-0.366 1.366)(1.8 3.1)(7.38 4.2)
\linewd 0.01
\clvec(7.4 4.2)(7.9 4.3)(8.4 4.4)
%%%%%%%%%%%%%%%
%% pjump sotto
%%%%%%%%%%%%%%%
\linewd 0.03
\move (-0.866 -0.5)\clvec(1.8 -4.9)(-0.366 -1.366)(7.95 -3)
\linewd 0.01
\clvec(8 -3.01)(8.5 -3.1)(9 -3.2)
%%%%%%%%%%%%%%%
%% vjump sopra
%%%%%%%%%%%%%%%
\linewd 0.03
\move (1 0)\clvec(1.75 0.5)(1.8 1.1)(7.86 3.2)
\linewd 0.01
\clvec(7.88 3.21)(8.38 3.38)(8.86 3.46)
%%%%%%%%%%%%%%%
%% vjump sotto
%%%%%%%%%%%%%%%
\linewd 0.03
\move (1 0)\clvec(4 -2)(7.1 -2.2)(8.14 -2.42)
\linewd 0.01
\clvec(8.6 -2.5)(8.6 -2.49)(9.1 -2.61)
%%%%%%%%%%%%%%%%
%% archi grassi
%%%%%%%%%%%%%%%%
\linewd 0.03
\move (0 0)\larc r:8.5 sd:22.16 ed:29.6
\move (0 0)\larc r:8.5 sd:339.4 ed:343.45
%%%%%%%%%%%%%%%%
%% frecce
%%%%%%%%%%%%%%%%
\linewd 0.03
\move (3.5 3.174) \arrowheadtype t:V  \arrowheadsize l:0.3 w:0.2 \avec (3.45 3.158)
\move (3.5 1.524) \arrowheadtype t:V \arrowheadsize l:0.3 w:0.2 \avec (3.45 1.5)
\move (7.6 3.8)   \arrowheadtype t:V \arrowheadsize l:0.3 w:0.2 \avec (7.7 3.62)
\move (8.05 -2.7) \arrowheadtype t:V \arrowheadsize l:0.3 w:0.2 \avec (8 -2.85)
%%%%%%%%%%%%%%%%%%
%% scritte
%%%%%%%%%%%%%%%%%%
\textref h:C v:C
\htext (0 0) {$|x|\leq\varepsilon$}
\htext (9.1 0) {$|x|=R$}
\htext (4.3 1.3) {$x=\hat x(t)$}
\htext (3.7 3.7) {$x= x(t)$}
\htext (7.65 4.6) {$x(T_1)$}
\htext (7.3 -3.2) {$x(T_2)$}
\htext (8.5 3.0) {$\hat x(\hat T_1)$}
\htext (8.8 -2.2) {$\hat x(\hat T_2)$}
\htext (7.9 3.8) {$\gamma_1$}
\htext (8.4 -2.78) {$\gamma_2$}
\htext (9.5 4.5) {$-\xi^-$}
\htext (9.8 -2.8) {$\xi^+$}
}
\end{texdraw}
\end{mac}
\end{center}
\caption{ \label{fig:3} if $\gamma_i$ achieves $m\left(x(T_i),\hat x(\hat T_i),R  \right)$, $i=1,2$,
then equation \eqref{eq:unico_tipo_1} holds.}
\end{figure}
Lemma~\ref{lem:est_from_above} implies that
\begin{multline*}
m\left(x(T_1),\hat x(\hat T_1),R  \right)
\leq \frac{\pi\sqrt{2V_{\max}}}{2} R^{\alpha_*} |s(T_1)-\hat s(\hat T_1)| \\
\leq \frac{\pi\sqrt{2V_{\max}}}{2} R^{\alpha_*} \left[ |s(T_1)-\xi^-| + |\hat s(\hat T_1)-\xi^-|\right],
\end{multline*}
and the last term goes to zero by Lemma \ref{lem:asintotiche}.
Since an analogous estimate holds for $m(\hat x(\hat T_2),x(T_2),R )$,
we conclude by exchanging the role of $x$ and $\hat x$.
\end{proof}
We are finally ready to prove the main result of this section.
\begin{proposition}\label{propo:unico_tipo}
Let $\ve>0$ and $V \in \Poh$ be fixed and let $x,\, \hat x \in \morse$.
Then
\[
\pjump(x) = \pjump(\hat x) \quad \text{and} \quad \vjump(x) = \vjump(\hat x).
\]
\end{proposition}
\begin{proof}
We recall that, for each trajectory, at least one jump must vanish.
Let us start considering the case in which
\[
\pjump(\hat x)=\vjump(x)=0.
\]
Taking into account the estimates in Lemmata \ref{lem:stima_sup} and \ref{lem:stima_inf}
we obtain that
\begin{multline*}
\action\left([T_1,T_2];x\right) - \action\left([\hat T_1,\hat T_2]; \hat x \right) \geq \\
- \frac{R^{\alpha_*}}{\alpha_*}
\left[ \sqrt{2V(\hat s(\hat T_2))} + \sqrt{2V(\hat s(\hat T_1))} - 2\sqrt{2V_{\min}} \right]\\
+ \ve^{\alpha_*} \left[ \sqrt{2V_{\min}}\pjump(x) +\frac{1}{\alpha_*}\vjump(\hat x) \right]
-\frac{2}{\alpha_*} \sqrt{2V_{\min}}\ve^{2\alpha_*}R^{-\alpha_*}.
\end{multline*}
By Lemma \ref{lem:unico_tipo} we have that the left hand side above goes to 0 as
$R \to +\infty$. By rearranging, we obtain, for a suitable $C>0$,
\begin{multline}\label{eq:unico_tipo_2}
C\left[\pjump(x)+\vjump(\hat x)\right]\ve^{\alpha_*}\\
 \leq  R^{\alpha_*}\left[ \sqrt{V(\hat s(\hat T_2))} - \sqrt{V_{\min}} + \sqrt{V(\hat s(\hat T_1))} - \sqrt{V_{\min}} \right] + o(1).
\end{multline}
Since when $t$ is sufficiently large (positive or negative),
\[
\sqrt{V(\hat s(t))} - \sqrt{V_{\min}} \leq  C|\hat s(t) - \xi^\pm|^2,
\]
using Lemma \ref{lem:asintotiche} we infer that,
choosing $R$ sufficiently large, the right hand side of the previous
inequality can be made arbitrarily small.
As a consequence, since the (nonnegative) left hand side of equation \eqref{eq:unico_tipo_2}
does not depend on $R$, we have that it must vanish, implying that
$\pjump(\hat x) = \vjump(x)= \pjump(x) = \vjump(\hat x) = 0$.

Coming to the case in which
\[
\pjump(x) = \pjump(\hat x) =0,
\]
one can argue exactly as above obtaining, instead of \eqref{eq:unico_tipo_2}
\[
C\left|\vjump(x)-\vjump(\hat x)\right|\ve^{\alpha_*} \leq  o(1),
\]
and concluding that $\vjump(x)=\vjump(\hat x)$.
Finally the case in which $\vjump(x) = \vjump(\hat x) =0$ can be ruled out in the same way.
\end{proof}
\begin{remark}\label{rem:hom2}
Given $V \in \Poh$, homotheticity induces
a one-to-one correspondence between
the sets of $\ve_1$-constrained Morse minimizers
and of $\ve_2$-constrained Morse minimizers, for any $\ve_1$, $\ve_2$.
Moreover the quantities $\pjump$ and $\vjump$ are invariant with respect
to this correspondence (see Remark \ref{rem:hom1}).
As a consequence, the jumps are independent not only of $x \in \morse$,
but also of $\ve$, and they depend only on $V$.
\end{remark}
Motivated by the previous remark we extend the definition of jumps to the potentials.
\begin{definition}\label{def:jumpV}
Let $V \in \Poh$, then
\begin{itemize}
\item $\pjump(V) := \pjump(x)$, for every $\eps>0$ and $x \in \morse$,
\item $\vjump(V) := \vjump(x)$, for every $\eps>0$ and $x \in \morse$.
\end{itemize}
\end{definition}
To conclude this section, we emphasize that Theorem \ref{teo:stab} allows us
to characterize $\pjump(V)$, $\vjump(V)$ in terms of approximating Bolza minimizers
and/or potentials.
\begin{remark}\label{rem:hom4}\label{rem:hom5}
Reasoning as in the proof of Lemma \ref{lem:Morse_not_empty},
we obtain that, if $V_n$ is a sequence in $\Poh$ with $V_n \to V$,
and $x_n$ is a Bolza $\ve$-constrained minimizer for $V_n$
(such that their endpoints satisfy $|x^{\pm}_{n}| \to +\infty$ and $x^{\pm}_{n} / |x^{\pm}_{n}| \to 	\xi^\pm$),
then, up to subsequences, $x_n$ converges to a Morse minimizer for $V$, and
\[
\lim_n \pjump(x_n) = \pjump(V).
\]
In particular,
\[
\lim_n \pjump(V_n) = \pjump(V)
\]
(similar equalities can be obtained for velocity jumps).
As a consequence, the set $\Poh$ can be written as disjoint union in the following way
\[
 \Poh = \{V: \pjump(V)>0\} \cup \{V: \vjump(V)>0\} \cup \{V: \pjump(V)=\vjump(V)=0\},
\]
where the first two sets are open.
\end{remark}
We remark that, up to now, the three sets above do not need to be all non-empty.
Furthermore, let us focus on potentials $V$ such that $\pjump(V)=\vjump(V)=0$;
while any corresponding constrained Morse minimizer
is a solution of the Euler-Lagrange equation \eqref{eq:dynsys} on the whole real line,
on the other hand at this moment we do not know whether, as we expect, it is also a free parabolic minimizer.

All these questions will find a positive answer in the next section.

%==============================
\section{Parabolic Minimizers}\label{sec:ParMin}
%==============================
%
The aim of this section is to investigate the relations between
the sets $\setin$, $\setout$, defined in the Introduction,
and the classification of potentials in terms of the jumps of their minimizers
(see Remark \ref{rem:hom4}). As a byproduct we will obtain all the main results we
presented in the Introduction.

We recall that the sets $\setin$ and $\setout$ were defined in terms of the (non constrained) Bolza level
\[
c(V) = \inf \left\{ \action\left([a,b];x\right) : a<b, \, x \in H^1(a,b), \, x(a)=\xi^-,\, x(b)=\xi^+  \right\}.
\]
It is then natural to go back to the tools introduced in Section \ref{sec:Bolza}
where we studied the constrained Bolza minimization problem \eqref{eq:minA}.
In the present situation we have that $x_1=\xi^-$ and $x_2=\xi^+$ are fixed.
On the contrary, besides the one on the constraint, we also want to consider the dependence on the potential. Accordingly,
we change our notation and, for any $V \in \Poh$ and $\ve <1$, we write
\[
m(V,\ve) :=  m(\xi^-,\xi^+,\ve)
\]
to denote the action level of $\ve$-constrained
minimizers connecting $\xi^-$ to $\xi^+$, obtaining that
\[
c(V) = \inf_{\eps \in (0,1]} m(V,\eps).
\]
We recall that (see Appendix \ref{app:stime})
\[
m(V,0) = 2\, \mathrm{hom}\left( 0, 1, V_{\min}\right)
              = \frac{2}{\alpha_*}\sqrt{2V_{\min}}.
\]
In what follows, a central role is played by the function $\gamma$
that we introduce in the following lemma.
\begin{lemma}\label{lem:gamma}
Let $\gamma : \, \Poh \times (0,1] \to \RR$ be defined as
\[
\gamma(V,\ve) :=  \frac{m(V,\ve)-m(V,0)}{\ve^{\alpha_*}}.
\]
Then:
\begin{enumerate}
	\item[(1)] $\gamma$ is continuous;
	\item[(2)] for every fixed $V$, $\gamma(V,\cdot)$ is increasing on $(0,1]$;
  \item[(3)] $\gamma(V,0^+): \, \Poh \to \{-\infty\}\cup\RR$ is (well defined and) upper semi-continuous.
\end{enumerate}
\end{lemma}
\begin{proof}
(1) The continuity of $\gamma$ follows from the one of $m$ (on $\Poh \times [0,1]$)
which in turn is a standard consequence of the stability theorem of Appendix \ref{app:stab}.
We give a brief sketch of this last argument.
Let $(V_n,\ve_n) \to (V,\ve)$ and let $x_n$ achieve $m(V_n,\ve_n)$.
With a suitable time translation we can apply Theorem \ref{teo:stab}
(separately on suitable time intervals $[T_{1,n},t_{*,n}]$, $[t_{*,n},t_{**,n}]$,
$[t_{**,n},T_{2,n}]$), obtaining that $x_n\to \bar x$, and furthermore
$m(V_n,\ep_n)\to\action_V(\bar x)\geq m(V,\ep)$. If, by contradiction, the strict inequality holds,
it is possible to use $\hat x$, achieving $m(V,\ep)$, to construct a test path which strictly decreases
the value of $m(V_n,\ep_n)$, for $n$ sufficiently large.

(2) Let $V$ be fixed, and let $\bar x$ achieve $m(V,\ve_0)$. For any $\ve \in (0,\ve_0)$ we consider
\[
\bar x_{{\ve}/{\ve_0}}(t) = \frac{\ve}{\ve_0} \bar x \left( \left(\frac{\ve}{\ve_0}\right)^{\frac{2+\alpha}{2}}t \right),
\]
which connects ${\ve}\xi^-/{\ve_0}$ to ${\ve}\xi^+/{\ve_0}$.
Let us denote by $y_{\ve}$ the juxtaposition
of $\bar x_{{\ve}/{\ve_0}}$ with the two homothetic motions joining its
endpoints with $\xi^\pm$.
By uniqueness of the Cauchy problem we have that $y_{\ve}$ is not $\cont^1$, hence
\begin{equation}\label{eq:come_vuoi}
\begin{split}
m(V,\ve) & < \action \left( y_\ve \right)
         = 2\, \mathrm{hom}\left( \frac{\ve}{\ve_0}, 1, V_{\min}\right) + \action(\bar x_{{\ve}/{\ve_0}}) \\
       & = 2\, \left[\mathrm{hom}\left( 0, 1, V_{\min}\right) -
                     \mathrm{hom}\left( 0,\frac{\ve}{\ve_0}, V_{\min}\right) \right]
                              + \left( \frac{\ve}{\ve_0}\right)^{\alpha_*} \action\left(\bar x \right)\\
       & = m(V,0) + \left( \frac{\ve}{\ve_0}\right)^{\alpha_*}\left( m(V,\ve_0)-m(V,0) \right),
\end{split}
\end{equation}
which implies  $\gamma(V,\ve) < \gamma(V,\ve_0)$.

(3) By the already proved monotonicity we have that $\gamma(V,0^+)$ is well defined for every $V$
and that
\[
\gamma(V,0^+) = \inf_{\ve>0} \gamma(V,\ve).
\]
Therefore the upper semi-continuity of $\gamma(V,0^+)$ follows from the continuity
of $\gamma(V,\ve)$, $\ve>0$, with respect to $V$.
\end{proof}
We are now ready to prove Proposition \ref{propo:alter}.
\begin{proof}[Proof of Proposition \ref{propo:alter}.]
Suppose that $\bar x$ achieves $c(V)$ and that, for some $t$, $\bar x(t)=0$.
From Lemma \ref{lem:omo2} it follows straightforwardly
that $\bar x$ is the juxtaposition of two homothetic motions, the first connecting $\xi^-$ to
the origin and the second the origin to $\xi^+$. In particular $c(V)=2\hom(0,1,V_{\min})$.

On the other hand let us assume that $\bar x$, with $\min_t|\bar x(t)|=:\bar \ve>0$,
achieve $c(V)$, and let  us assume by contradiction that
$c(V) =  2\hom(0,1,V_{\min}) = m\left(V,0\right)$.
By definition we have, for every $\ve$, $c(V) = m(V,\bar\ve) \leq m(V,\ve) $
that implies
\[
\gamma(V,\ve) \geq \gamma(V,\bar\ve)=0
\]
in contradiction with the strict monotonicity of $\gamma$ (see Lemma \ref{lem:gamma}).

To conclude, we deduce the existence of a lower bound for the norm
of minimizers arguing by contradiction and using Theorem \ref{teo:stab}.
\end{proof}
We can give an equivalent definition of the sets $\setin$ and $\setout$ in terms of $\gamma$.
\begin{lemma}\label{lem:inout}
Let $\setin$ and $\setout$ be defined as in the Introduction.
Then
\[
	\setin = \left\{ V \in \Poh :\, c(V) = m(V,0)  \right\}
	       = \left\{ V \in \Poh :\, \gamma(V,0^+)\geq 0\right\},
\]
and
\[
	\setout = \left\{ V \in \Poh :\, c(V) < m(V,0)  \right\}
          =	\left\{ V \in \Poh : \gamma(V,0^+)< 0\right\}.
\]
Moreover, $\setin$ is closed while $\setout$ is open.
\end{lemma}
\begin{proof}
Since $c(V)=\inf_\ve m(V,\ve)$, we have that
\[
c(V) < m(V,0) \quad \iff \quad
\gamma(V,\bar\eps)<0, \text{ for some } \bar \ve>0.
\]
Recalling Lemma \ref{lem:gamma} we obtain the equivalent definitions of $\setin$ and $\setout$.
Finally, their topological attributes are a consequence of the upper semi-continuity of $\gamma(V,0^+)$.
\end{proof}
As we mentioned, the decomposition of $\Poh$ in terms of $\setin$ and $\setout$ is strictly related
to the one induced by the jumps of the potentials (Remark \ref{rem:hom5}).
\begin{lemma}\label{lem:caratt_in_out}
Let $V \in \Poh$. Then
\[
V \in \setin \iff \vjump(V)=0 \qquad \text{(resp. $V \in \setout \iff \vjump(V)>0$).}
\]
\end{lemma}
\begin{proof}
For any $n \geq 1$ let $\bar x_n$ be any Bolza minimizer achieving $m(V,1/n)$,
and let $x_n(t) = n \bar x_n(n^{-(2+\alpha)/{2}}t)$.
On one hand, by Remark \ref{rem:hom1}, we have that
$\pjump(x_n) = \pjump(\bar x_n)$ and $\vjump(x_n) = \vjump(\bar x_n)$;
on the other hand, recalling the proof of Lemma \ref{lem:Morse_not_empty},
we have that $x_n$ converges to a constrained Morse minimizer.
This allows us to relate the jumps of $\bar x_n$ with the ones of $V$
(recall Remark \ref{rem:hom4}).

Now let us assume that $V \in \setin$, and, by contradiction, that
$\pjump(\bar x_n)=0$ (a contradiction will imply that $0=\vjump(\bar x_n) \to \vjump(V)$).
Then we can use Lemma \ref{lem:stima_sup} obtaining
\[
m(V,1/n) \leq m(V,0) - \frac{1}{\alpha_*} \vjump(\bar x_n) \frac{1}{n^{\alpha_*}}
\]
hence
\[
\gamma (V,0^+) < \gamma (V,1/n) \leq - \frac{1}{\alpha_*} \vjump(\bar x_n) \leq 0,
\]
in contradiction with the definition of $\setin$, see Lemma \ref{lem:inout}.

Analogously, let $V \in \setout$, and, by contradiction, suppose that
$\vjump(\bar x_n)=0$. On one hand, Lemma \ref{lem:stima_inf} applies, yielding
\[
\gamma (V,1/n) \geq \sqrt{2V_{\min}} \left( \pjump(\bar x_n) -\frac{2}{\alpha_*}\frac{1}{n^{\alpha_*}}\right);
\]
on the other hand, since $c(V)=m(V,\bar \eps)$ for some $\bar \eps>0$,
for $n$ large, equation \eqref{eq:come_vuoi} holds so that
\[
\gamma (V,1/n)  < \frac{c(V)-m(V,0)}{\bar \ve^{\alpha_*}} < 0, \quad \text{independent on $n$}.
\]
For $n$ large the last two inequalities imply $\pjump(\bar x_n)<0$, a contradiction. Resuming, we have proved that
\[
V \in \setout \qquad \implies \qquad \pjump(\bar x_n)=\pjump(V)=0;
\]
therefore we are allowed to apply Lemma \ref{lem:stima_sup} in order to obtain
\[
\gamma (V,1/n)  \geq -\frac{1}{\alpha_*}\vjump(x_n) - \frac{2\sqrt{2V_{\min}}}{\alpha_*} \frac{1}{n^{\alpha_*}}.
\]
Passing to the limit, as $n \to +\infty$, we finally obtain
$\vjump(V) \geq -\alpha_* \gamma (V,0^+)>0$, indeed, by assumption,
$V \in \setout$ and hence $\gamma (V,0^+)<0$.
\end{proof}
\begin{corollary}\label{coro:6.5}
Let $V$ be such that $\pjump(V)=\vjump(V)=0$, $\ve>0$ be small and $x_{\ve}$ achieve $m(V,\ve)$.
What we actually proved is that both $\vjump(x_\ve)=0$ and $0<\pjump(x_\ve) \to 0$.
\end{corollary}
\begin{corollary}\label{coro:6.4}
Let $V \in \Pi:= \partial{\setin} \cap \partial{\setout}$, then $\pjump(V)=\vjump(V)=0$
(recall Remark \ref{rem:hom5}).
\end{corollary}
Our aim is to prove that in fact $\Pi$
coincides with the set of potentials such that $\pjump(V)=\vjump(V)=0$.
A key result in this direction is the following one.
\begin{lemma}\label{lem:mono_gamma}
Let $V \in \Seh$, and $0<\alpha_1<\alpha_2<2$.
Then
\[
\gamma \left( (V,\alpha_2),0^+ \right) - \gamma \left( (V,\alpha_1),0^+ \right) \leq
-\frac{4\sqrt{2V_{\min}}}{(2-\alpha_1)^2}(\alpha_2-\alpha_1).
\]
\end{lemma}
\begin{proof}
Let us fix $\ve>0$ and let $\bar x = \bar r \bar s$, defined on $[0, \bar T]$ for some $\bar T$,
achieve $m((V,\alpha_1),\ve)$.
For any $\alpha \in (\alpha_1,2)$,
we consider the following  reparameterization of the path $\bar x$
\[
y_\alpha (\vartheta) := \bar x(t(\vartheta)), \quad
\text{where } \vartheta \in [0,T_\alpha],  \text{ for some $T_\alpha$, and }
\frac{\dtheta}{\dt} = \bar r^{(\alpha-\alpha_1)/2}.
\]
We have that $y_\alpha = r_\alpha s_\alpha$
satisfies (see Lemma \ref{lem:maup})
\[
\frac12 \left| y'_\alpha \right|^2 - \frac{V(s_\alpha)}{r_\alpha^{\alpha}} =0, \quad \text{on } [0,T_\alpha]
\]
(here `` $'$ '' denote the derivative with respect to $\vartheta$).
Let us define the function
\[
f(\alpha) := \frac{\action_\alpha(y_\alpha) - m((V,\alpha),0)}{\ve^{\alpha_*}},
\]
where
\[
\action_\alpha(y_\alpha) =
\int_0^{T_\alpha} \left(\frac12 \left| y'_\alpha \right|^2 +
\frac{V(s_\alpha)}{r_\alpha^{\alpha}}\right) \dtheta =
\int_0^{T_\alpha} \frac{2V(s_\alpha)}{r_\alpha^{\alpha}}\, \dtheta =
\int_0^{\bar T}   \frac{2V(\bar s)}{\bar r^{(\alpha+\alpha_1)/2}}\, \dt.
\]
Computing the derivative of $f$ we obtain
\[
f'(\alpha) = \frac{1}{\ve^{\alpha_*}}
             \left\{ -\frac12 \int_{0}^{\bar T} \frac{2V(\bar s)}{\bar r^{(\alpha+\alpha_1)/2}}
                                                 \log\frac{\bar r}{\ve}\,\dt
                     - \sqrt{2V_{\min}}\, \left( \frac{1}{\alpha_*^2}+\frac{\log\ve}{\alpha_*} \right) \right\}.
\]
In order to estimate $f'(\alpha)$ we remark that, since
${\sqrt{2V(\bar s)}}\,{\bar r^{-\alpha_1/2}}=|\dot{\bar x}|\geq |\dot{\bar r}|$,
the following inequality holds
\begin{multline*}
\int_{0}^{\bar T} \frac{2V(\bar s)}{\bar r^{(\alpha+\alpha_1)/2}} \log\frac{\bar r}{\ve}\,\dt \geq
\sqrt{2V_{\min}}\,
\int_{0}^{\bar T} \frac{\sqrt{2V(\bar s)}}{\bar r^{(\alpha+\alpha_1)/2}} \log\frac{\bar r}{\ve}\,\dt \\
\geq
\sqrt{2V_{\min}}\, \int_{0}^{\bar T} \frac{|\dot{\bar r}|}{\bar r^{\alpha/2}} \log\frac{\bar r}{\ve}\,\dt;
\end{multline*}
the convexity of ${\bar r}^2$ allows the change of variable
$r=\bar r(t)$ on the two monotonicity intervals of $\bar r(t)$,
obtaining
\[
\int_{0}^{\bar T} \frac{|\dot{\bar r}|}{\bar r^{\alpha/2}} \log\frac{\bar r}{\ve}\,\dt =
2\int_{\ve}^{1} \frac{1}{r^{\alpha/2}} \log\frac{r}{\ve}\,\dr =
2\left( -\frac{\log \ve}{\alpha_*} -\frac{1}{\alpha_*^2} +\frac{\ve^{\alpha_*}}{\alpha_*^2}\right);
\]
hence
\[
f'(\alpha) \leq -\frac{\sqrt{2V_{\min}}}{\alpha_*^2} = -\frac{4\sqrt{2V_{\min}}}{(2-\alpha)^2}.
\]
Now, since
$\action_{\alpha_1}(y_{\alpha_1}) = m((V,\alpha_1),\ve)$ and
$\action_{\alpha_2}(y_{\alpha_2}) \geq m((V,\alpha_2),\ve)$, we infer that,
for a suitable $\xi \in (\alpha_1,\alpha_2)$,
\[
\begin{split}
\gamma \left( (V,\alpha_2),\ve\right) - \gamma \left( (V,\alpha_1),\ve\right)
& \leq f(\alpha_2)-f(\alpha_1) = f'(\xi)(\alpha_2-\alpha_1) \\
& \leq -\frac{4\sqrt{2V_{\min}}}{(2-\xi)^2}(\alpha_2-\alpha_1)
  \leq -\frac{4\sqrt{2V_{\min}}}{(2-\alpha_1)^2}(\alpha_2-\alpha_1).
\end{split}
\]
The proposition follows by taking the limit as $\ve \to 0^+$.
\end{proof}
\begin{corollary}
Let us fix $V \in \Seh$. Both
\[
\gamma((V,0^+),0^+):= \lim_{\alpha \to 0^+} \gamma((V,\alpha),0^+)
\]
and
\[
\gamma((V,2^-),0^+):= \lim_{\alpha \to 2^-} \gamma((V,\alpha),0^+)
\]
are well defined and
\[
\gamma((V,2^-),0^+) = -\infty.
\]
\end{corollary}
\begin{proof}
The limits are well defined by monotonicity. Moreover, for every $\alpha \in (\alpha_1,\alpha_2)$,
there holds
\begin{multline*}
(\alpha_2-\alpha_1) \gamma((V,\alpha_1),0^+)
\geq \int_{\alpha_1}^{\alpha_2}\gamma((V,\alpha),0^+)\d\alpha\\
\geq (\alpha_2-\alpha_1) \gamma((V,\alpha_2),0^+)
     + C\int_{\alpha_1}^{\alpha_2}  \frac{\alpha_2-\alpha}{(2-\alpha)^2} \d\alpha. \qedhere
\end{multline*}
\end{proof}
\begin{corollary}\label{coro:defbaralpha}
Let $\Sigma := \left\{ V \in \Seh : \gamma((V,0^+),0^+)>0 \right\}$. Then
for every $V \in \Sigma$ there exists exactly one $(0,2)\ni\alpha := \bar \alpha(V)$
such that $\gamma((V,\bar \alpha(V)),0^+)=0$.
\end{corollary}
Now we are in a position to prove that all the different notions of parabolicity we introduced
are equivalent.
\begin{theorem}\label{teo:equivalent}
Let $V \in \Poh$. Then the following facts are equivalent:
\begin{enumerate}
\item $V \in \Pi$;
\item $V$ admits a free parabolic Morse minimizer;
\item any constrained Morse minimizer for $V$ is a free parabolic one;
\item $\pjump(V)= \vjump(V) = 0$;
\item $\gamma(V,0^+)=0$.
\end{enumerate}
\end{theorem}
\begin{proof}
$[1. \implies 2.]$. Let $(V,\alpha) \in \Pi$.
Then, on one hand $(V,\alpha) \in \setin$; on the other hand there exists
$(V_n,\alpha_n) \in \setout$ with $(V_n,\alpha_n) \to (V,\alpha)$.
As a consequence there exists $y_n$ such that
\[
y_n \text{ achieves } c(V_n)
\quad \text{and} \quad
0 < \ve_n := \min|y_n| \to 0.
\]
Therefore the renormalized paths $x_n(t) := \ep_n^{-1} y_n\left(\ep_n^{-(2+\alpha_n)/2}t\right)$ are free Bolza minimizers.
Using Theorem \ref{teo:stab} we conclude that $x_n$ converges to a free parabolic Morse minimizer for
the potential $(V,\alpha)$ (see also Remark \ref{rem:freeconv}).\\
$[2. \implies 3.]$. Let $x$ denote a free parabolic Morse minimizer for $V$.
By rescaling, we  can assume that $x$ is also a 1-constrained Morse minimizer, so that
$\pjump(V) = \vjump(V) = 0$.
On the other hand let $\hat x$ be any 1-constrained Morse minimizer.
Mimicking the proof of Proposition \ref{propo:unico_tipo}, for any (large) $n \in \NN$
we denote with $x^{\pm}_n$, $\hat x^{\pm}_n$ the points of the two trajectories on the sphere of radius $n$.
Going back to the notation of the previous sections, we denote with
$m(x^-_n,x^+_n,1)$ and $m(\hat x^-_n,\hat x^+_n,1)$ the actions of the restriction of the paths
to the ball of radius $n$.
As a consequence, Lemma \ref{lem:unico_tipo} applies providing
\begin{equation} \label{eq:diffaction}
\left| m(x^-_n,x^+_n,1) - m(\hat x^-_n,\hat x^+_n,1) \right| \to 0, \quad \text{as } n \to \infty.
\end{equation}
On the other hand let us define
\[
c(\hat x^-_n,\hat x^+_n) := \inf_{\ve >0} m(\hat x^-_n,\hat x^+_n,\ve),
\]
in such a way that $\hat x$ is a free minimizer if and only if $c(\hat x^-_n,\hat x^+_n)
=m(\hat x^-_n,\hat x^+_n,1)$, for every $n$.
Reasoning by contradiction we can assume that, for instance,
\[
c(\hat x^-_2,\hat x^+_2) \leq m(\hat x^-_2,\hat x^+_2,1) - k
\]
for some $k>0$.
By juxtaposition, we have the following estimate:
\[
\begin{split}
c(\hat x^-_n,\hat x^+_n)
& \leq m(\hat x^-_n,\hat x^-_2,2) + c(\hat x^-_2,\hat x^+_2) + m(\hat x^+_2, \hat x^+_n,2) \\
& \leq m(\hat x^-_n,\hat x^-_2,2) + m(\hat x^-_2,\hat x^+_2,1) - k + m(\hat x^+_2, \hat x^+_n,2) \\
& =    m(\hat x^-_n,\hat x^+_n,1) - k.
\end{split}
\]
Since $x$ is a free minimizer, using the previous estimate we obtain
\[
\begin{split}
m(x^-_n, x^+_n,1)
& \leq m(x^-_n,\hat x^-_n,n) + c(\hat x^-_n,\hat x^-_n) + m(\hat x^+_n, x^+_n,n) \\
& \leq m(\hat x^-_n,\hat x^+_n,1) - k + \underbrace{m(x^-_n,\hat x^-_n,n)  + m(\hat x^+_n, x^+_n,n)}_{(*)}.
\end{split}
\]
Using Lemma \ref{lem:asintotiche} we can prove that $(*) = o(1)$ as $n \to \infty$
(recall the proof of Lemma \ref{lem:unico_tipo}).
But then the last estimate is in contradiction with equation \eqref{eq:diffaction}.\\
$[3. \implies 4.]$. This is trivial since if $x$ is a free parabolic minimizer, then
it is also a $(\min |x|)$-constrained minimizer.\\
$[4. \implies 5.]$. Since $\vjump(V) = 0$, we have that $V \in \setin$;
by Corollary \ref{coro:6.5} we have that, if $x_\ve$ achieves $m(V,\ve)$ and
$\ve$ is small, then $\vjump(x_{\ve}) = 0$. We can then apply Lemma \ref{lem:stima_inf}
and let $\ve \to 0$, in order to obtain
\[
0 \leq \gamma(V,\ve) \leq \frac{\pi\sqrt{2V_{\max}}}{2} \, \pjump(x_{\ve}) \to 0.
\]
$[5. \implies 1.]$.  On one hand, $V \in \setin$ by definition.
On the other hand, letting $\alpha_n := \alpha+1/n$, we have that $(V,\alpha_n) \to (V,\alpha)$
and, by Lemma \ref{lem:mono_gamma}, $(V,\alpha_n) \in \setout$.
\end{proof}
\begin{corollary}
Combining the previous theorem with Corollary \ref{coro:defbaralpha} we have that
\[
\Pi = \{ (V,\alpha) : V \in \Sigma \text{ and } \alpha = \bar \alpha (V)\}
\]
and
\[
\setin = \{ (V,\alpha) : V \in \Sigma \text{ and } \alpha \leq \bar \alpha (V)\}.
\]
\end{corollary}
\begin{proof}[End of the proof of the main results.]
Summarizing, in order to end the proof of the results we stated in the Introduction,
we have to show that $\Sigma$ is an open subset of $\Seh$, that the function $\bar \alpha$ is continuous,
and that $\Sigma$ (and therefore $\Pi$) is not empty.

Let $V_0 \in \Sigma$ and $\bar\alpha_0:=\bar\alpha(V_0)>0$; by Lemma \ref{lem:mono_gamma}
$\gamma \left((V_0,\bar\alpha_0 /2),0^+\right)>0$,
then $(V_0,\bar\alpha_0 /2) \in \setin$ and $\pjump\left((V_0,\bar\alpha_0 /2)\right)>0$.
By continuity of $\pjump$ we infer the existence of a neighborhood
$\mathcal{V}\times (\bar\alpha_0 /2-\delta,\bar\alpha_0 /2+\delta)$
such that, for any $(V,\alpha)$ in that neighborhood, $\pjump(V,\alpha)>0$.
As a consequence $\mathcal{V} \subset \Sigma$ is a neighborhood of $V_0 \in \sphere$.
Once we have proved that $\Sigma$ is open, the continuity of $\bar \alpha$
descends from the one of $\pjump$ and $\vjump$.
Finally the fact that $\Sigma$ is not empty is implied by Lemma \ref{lem:O} below.
\end{proof}
We conclude providing some explicit conditions to ensure that a $V \in \Seh$
belongs to $\Sigma$. Roughly speaking, this happens when $V$ is very much larger
than $V_{\min}$ on a zone that must be crossed in order to connect the two minimal
configurations $\xi^-$ and $\xi^+$.
\begin{lemma}\label{lem:O}
Let $V \in \Seh$ and let us assume that the open set $O \subset \sphere$ is such that
\begin{enumerate}
\item $\sphere \setminus O$ has exactly two connected components $F^-$ and $F^+$
      with $\xi^{\pm} \in F^{\pm}$.
\item $\displaystyle \sqrt{2\min_{s \in \overline O}(V(s)-V_{\min})} \cdot \mathrm{dist}(F^-,F^+) > 2\sqrt{2V_{\min}}$.
\end{enumerate}
Then $V \in \Sigma$.
\end{lemma}
\begin{proof}
Let $\alpha \in (0,2)$ and $\ve \in (0,1)$ be fixed.
Furthermore let $x=rs$ be any path joining $\xi^-$ and $\xi^+$, with $\min r = \ve$;
by assumption, there exists an interval $[t_1,t_2]$ of times in such a way that
$s((t_1,t_2)) \subset O$, $s(t_1) \in \partial F^-$ and $s(t_2) \in \partial F^+$.
As a consequence we have that $\mathrm{dist}(F^-,F^+) \leq |s(t_2)-s(t_1)|$ and thus
\[
\sqrt{2\min_{t \in [t_1,t_2]}(V(s(t))-V_{\min})} \cdot |s(t_2)-s(t_1)|
\geq 2\sqrt{2V_{\min}}+k,
\]
for a suitable $k>0$.
Using this information in Lemma \ref{lem:est_from_below} we obtain that
\[
\begin{split}
m((V,\alpha),\ve)
& \geq 2\hom(\ve,1,V_{\min}) + (2\sqrt{2V_{\min}} + k)\ve^{\alpha_*}\\
& = m((V,\alpha),0) + \left( -\frac{2\sqrt{2V_{\min}}}{\alpha_*}+ 2\sqrt{2V_{\min}} + k\right)\ve^{\alpha_*},
\end{split}
\]
implying
\[
\gamma((V,\alpha),\ve) \geq -\frac{2\sqrt{2V_{\min}}}{\alpha_*}+ 2\sqrt{2V_{\min}} + k.
\]
The lemma follows by taking the limit of the previous inequality first with respect to $\ve \to 0^+$
and then with respect to $\alpha \to 0^+$ (recall the definition of $\Sigma$ in Corollary \ref{coro:defbaralpha}).
\end{proof}
Of course, reasoning as in the previous lemma, it is possible to
manage also situations where the set $O$, on which $V$ is larger than $V_{\min}$,
disconnects $\sphere$ in more than two components.
We remark that to fulfill the assumptions of Lemma \ref{lem:O},
two (slightly different) mechanisms are available: either one can act locally near $\xi^\pm$,
e.g. choosing $\mu$ sufficiently large in the definition of $\Seh$;
or the potential  can be chosen arbitrarily ``flat'' near $\xi^\pm$ and
very large elsewhere.

To conclude, exploiting the characterization we obtained, we are in a position to
clarify the meaning of Remark \ref{rem:weakkam}.
\begin{remark}\label{rem:hj}
Consider the planar case with $\xi^-=\xi^+=\xi$, under a suitable topological constraint, and assume $V\in\Pi$. As already remarked in the Introduction, associated with the
family of the free parabolic minimizers there is a lamination of the plane.
This lamination  inherits an interesting minimization property: indeed,
let $\overline x(t)=\overline\rho(t) \xi$ be the homothetic  ejection  trajectory (Appendix \ref{app:stime})
such that $\overline\rho(0)=0$ and
$\overline\rho(+\infty)=+\infty$ and consider the map $w:\RR^2\setminus\{0\}\to\RR$
defined as
\begin{multline*}
w(z) = \min\left\{\int_0^{+\infty}[\lagr(\dot x(t),x(t))-\lagr(\dot{\overline{x}}(t),\overline{x}(t))]\dt,
\text{ such that} \right.\\
\left. x(0)=z \text{ and } \int_0^{+\infty}\vert x(t)-\overline x(t)\vert^2 \dt < +\infty \right\}.
\end{multline*}
Then, at each $z$, the minimal path is one of the two infinite arcs of the unique minimal parabolic trajectory $z(t)$ passing through $z$ at $t=0$. The contribution of the (renormalized) action can be easily computed and it is exactly $\pm z(0)\cdot \dot z(0)$ (Lemma \ref{lem:stima_azione}).
Hence there is a unique minimizing path passing through $z$ whenever the derivative of the radius does not vanish. Observe that the function $w$ is of class $\mathcal C^1$ except at the points where  $z(0)\cdot \dot z(0)=0$, where the minimizing arc is not unique. The gradient of $w$ is the velocity $\pm\dot z(0)$, depending on the orientation. Hence, the function $w$ is  Lipschitz continuous and solves the stationary Hamilton-Jacobi equation
\[
\dfrac12|\nabla u(z)|^2-V(z)=0,
\]
except at the points where $z(0)\cdot \dot z(0)=0$.
A global $\mathcal C^1$ solution can be easily construct on the double covering of the punctured plane (this is reminiscent of the double covering associated with Levi-Civita regularization).
\end{remark}

%==============================
\begin{appendices}
%==============================
For the reader's convenience, we collect here the proofs of some slight modifications
of rather standard arguments.
%==============================
\section{The Maupertuis' Principle}\label{app:maup}
%==============================
Let us consider the (sufficiently smooth) maps
\begin{itemize}
 \item $K:\RR^{2d}\ni(\dot x,x)\mapsto K(\dot x,x)\in[0,+\infty)$, with $K(\lambda\dot x,x)= \lambda^2K(\dot x,x)$, for every $\lambda$, and
 \item $P:\RR^{d}\ni x \mapsto P(x)\in(0,+\infty]$.
\end{itemize}
\begin{lemma}\label{lem:maup}
Let $x\in H^1\left( (a,b); \RR^d\right)$ be a fixed path, such that
\[
A([a,b];x) := \int_a^b \left[ K(\dot x(t),x(t))+P(x(t))\right] \dt\in(0,+\infty)
\]
and
\[
K(\dot x(t),x(t))\geq \delta >0 \text{ on }(a,b).
\]
Let us consider the set of all re-parameterizations of $x$, namely
\[
\Gamma_{x}:=\left\{\begin{array}{ll}
                     ((0,T),f) : & \,f:(0,T)\to(a,b),\,\text{Lipschitz continuous and}\\
                      &\text{increasing, such that } x\circ f\in H^1\left( (0,T); \RR^d\right)
                   \end{array}
            \right\}.
\]
Finally, let us define
\[
\vartheta(t):=\int_a^t\sqrt{\frac{K(\dot x(t),x(t))}{P(x(t))}}\dt,\quad
\hat T:=\int_a^b\sqrt{\frac{K(\dot x(t),x(t))}{P(x(t))}}\dt,
\]
\[
\hat f := \vartheta^{-1} :[0,\hat T]\to [a,b].
\]
Then
\[
\min_{\Gamma_x} A([0,T],x\circ f)\text{ is achieved by }\left(\left(0,\hat T\right),\hat f\right).
\]
Moreover, writing $\hat x := x\circ \hat f$, we have that, for (almost) every $\vartheta$,
\[
K({\hat x}'(\vartheta),{\hat x}(\vartheta))=P({\hat x}(\vartheta)) \quad \left( ':= \frac{\d}{\d\vartheta} \right).
\]
\end{lemma}
\begin{proof}
Let us observe that, for any $f \in \Gamma_x$,
\begin{multline*}
A([0,T],x\circ f)  =
\int_0^T \left[K(\dot x(f(\vartheta))f'(\vartheta),x(f(\vartheta)))+P(x(f(\vartheta)))\right] \d\vartheta \\
                   =
\int_0^T \left[f'(\vartheta)^2 K(\dot x(f(\vartheta)),x(f(\vartheta)))+P(x(f(\vartheta)))\right] \d\vartheta \\
 \geq
2\left(
\int_0^T f'(\vartheta)^2 K(\dot x(f(\vartheta)),x(f(\vartheta))) \d\vartheta \cdot
\int_0^T P(x(f(\vartheta))) \d\vartheta \right)^{1/2}\\
 \geq
2\int_0^T  \sqrt{K(\dot x(f(\vartheta)),x(f(\vartheta))) \cdot P(x(f(\vartheta)))} f'(\vartheta)\d\vartheta \\
 =
2\int_a^b  \sqrt{K(\dot x(t),x(t)) \cdot P(x(t))} \dt,
\end{multline*}
and equality holds if and only if
\[
f'(\vartheta)^2 K(\dot x(f(\vartheta)),x(f(\vartheta))) = P(x(f(\vartheta)))
\]
almost everywhere. Since the last term in the previous inequality does not depend on $f$, we have that
$f$ minimizes $A([0,T],x\circ f)$ if and only if the last equality holds. This is equivalent to
satisfy
\[
f'(\vartheta) = \sqrt{\frac{P(x(f(\vartheta)))}{K(\dot x(f(\vartheta)),x(f(\vartheta)))}}.
\]
Since $f$ is strictly increasing, we can use its inverse in order to write $\vartheta = \vartheta(t)$,
and the lemma follows.
\end{proof}
\begin{corollary}\label{lem:generale}
Let $\Gamma$ be a set of paths closed under re-parametrization and
let $\bar x \in \Gamma$ be such that $A(\bar x) =\min_{x \in \Gamma}A(x)$.
Then, for (almost) every $t$,
\[
K(\dot {\bar x}(t),{\bar x}(t))=P({\bar x}(t)).
\]
\end{corollary}
%

%==============================
\section{A Stability Theorem}\label{app:stab}
%==============================
\begin{theorem}\label{teo:stab}
Let us assume that, for $n \in \NN$,
\begin{enumerate}
\item $\alpha_n \to \alpha \in (0,2)$, $\liminf_n T_{1,n} = T_1$, $\limsup_n T_{2,n} = T_2$;
\item $V_n \in \Poh$ is $\alpha_n$-homogeneous,
$V \in \Poh$ is $\alpha$-homogeneous, and
$V_n|_{\sphere} \to V|_{\sphere}$ in $\cont^1(\sphere)$;
\item $z_n \in \cont^2(T_{1,n},T_{2,n};\RR^d \setminus B_{\ve}(0))$ satisfies
$ \ddot z_n = \nabla V_n(z_n)$;
\item (up to time translations) $T_{1,n}< \bar t_n <T_{2,n}$ and
$|z_n(\bar t_n)| + |\dot z_n(\bar t_n)| \leq C$, for some $\bar t_n \to \bar t$, and $C>0$.
\end{enumerate}
Then there exists a subsequence $\suc{z_n}{k} \subset \suc{z}{n}$
and a function $\bar z \in \cont^2(T_1,T_2)$ such that
$z_{n_k}|_I$ converges in $\cont^2$ to $\bar z|_I$,
for every $I \subset (T_1,T_2)$ compact.
\end{theorem}
\begin{proof}
First of all let us observe that there exists a constant $C>0$
such that, for every $n$,
\[
\left| \nabla V_n(x) \right| \leq C, \quad \forall x \in \RR^d \setminus B_{\ve}(0).
\]
Let $k \in \NN$ and $I_k = [\bar t-k,\bar t+k] \cap [T_1,T_2]$.
We infer that, up to a subsequence, $\ddot z_n$ is
(defined and) bounded on $I_k$. Integrating assumption 4. we obtain that
$|z_n| + |\dot z_n| \leq C$ on $I_k$. Ascoli's Theorem guarantees that, again
up to a subsequence, there exists $\bar z$ such that $z_n \to \bar z$ in $\cont^1(I_k)$.
Passing to the limit in the equations we have that the convergence is indeed $\cont^2$,
and that $\bar z$ satisfies the limiting equation. By a diagonal procedure we easily conclude.
\end{proof}

%==============================
\section{Properties of Zero-Energy Trajectories}\label{app:stime}
%==============================
The aim of this appendix is to sum up some well known results about
the behavior of zero energy trajectories.
The first ones concern homothetic trajectories, which are motions
with constant angular part.
\begin{lemma}\label{lem:omo1}
Let us fix $\gamma >0$ and consider the functional
\[
 \action_{\mathrm{rad},\gamma} \left([a,b],r\right) :=
  \int_{a}^{b}\left[ \frac12 \dot r^2(t) + \frac{\gamma}{r^\alpha(t)} \right] \dt
\]
defined on $H^1\left((a,b);[0,+\infty)\right)$.
Then, for any $r_+ \geq r_- \geq 0$,
\[
\begin{split}
 \mathrm{hom} (r_-,r_+,\gamma)
& := \inf\left\{ \action_{\mathrm{rad},\gamma} \left([-T,T],r\right) :\,
        \begin{array}{l}
          T\geq0, r \in H^1(-T,T),\smallskip\\
          r(\pm T)=r_{\pm}
        \end{array}
     \right\}\\
& =  \frac{\sqrt{2\gamma}}{\alpha_*}\left( r_+^{\alpha_*} - r_-^{\alpha_*}\right).
\end{split}
\]
\end{lemma}
\begin{proof}
If $r_-=r_+$ then the result is trivial.
Otherwise, arguing as in Section \ref{sec:Bolza} we deduce the existence of a monotone increasing minimizer
$\bar r$. From Corollary \ref{lem:generale} we deduce that
\[
\frac12 \dot {\bar r}^2(t) = \frac{\gamma}{\bar r^\alpha (t)}
\quad \implies \quad
\dot {\bar r}(t) = \sqrt{2\gamma}\bar r^{-\alpha/2} (t).
\]
Integrating the last equation and imposing the boundary conditions
we obtain the explicit expression of $\bar r(t)$
\[
\bar r(t) =
\left[\frac{\alpha+2}{2}\sqrt{2\gamma}\,t +
\frac12\left(r_+^{(2+\alpha)/{2}} + r_-^{(2+\alpha)/{2}}\right)\right]^{{2}/({2+\alpha})}
\]
which is defined on $[-\bar T,\bar T]$, where
$\bar T = \left[(\alpha+2)\sqrt{2\gamma}\right]^{-1} \left(r_+^{(2+\alpha)/{2}} - r_-^{(2+\alpha)/{2}}\right)$.
\end{proof}
\begin{lemma}\label{lem:omo2}
Let $\xi \in \sphere$ and $r_+ \geq r_- \geq 0$.
Then
\[
\begin{split}
\mathrm{hom} (r_-,r_+,V_{\min}) & \leq
\inf
 \left\{ \action([-T,T];x):\,
    \begin{array}{l}
         T>0, \, x \in H^1(-T,T),\smallskip\\
         x({\pm}T)= r_{\pm} \xi
    \end{array}
 \right\}\\
& \leq \mathrm{hom} (r_-,r_+,V(\xi)).
\end{split}
\]
In particular, if $V(\xi)=V_{\min}$, then equality holds (and the infimum is
achieved by a path with constant direction $\xi$).
\end{lemma}
\begin{proof}
The estimate from below follows straightforwardly from the previous lemma,
once one notices that, for any $x$ satisfying the constraint,
\[
 \action(x) = \int_{a}^b \left[ \frac12 \dot r^2 + \frac12 r^2 |\dot s|^2  +\frac{V(s)}{r^\alpha}\right]                \geq \int_{a}^b \left[ \frac12 \dot r^2 + \frac{V_{\min}}{r^\alpha}\right].
\]
On the other hand,
\[
x(t) = \bar r (t)\xi, \qquad t \in [-\bar T,\bar T],
\]
where $\bar r (t)$ and $\bar T$ have been defined in the proof of the
previous lemma, satisfies the constraint, providing the estimate from above.
\end{proof}
\begin{lemma}
Let us suppose that $x=rs$ satisfies both
\eqref{eq:dynsys} and \eqref{eq:dynsys2} on $(t_0,+\infty)$. Then
\begin{enumerate}
\item $r(t) \to +\infty$ and $\dot r(t)>0$, as $t \to +\infty$;
\item $\dot r(t) \to 0$ and $|\dot s(t)| \to 0$, as $t \to +\infty$.
\end{enumerate}
\end{lemma}
\begin{proof}
The first part follows
from the (strict) convexity of $r^2(t)$, which is implied
by the Lagrange-Jacobi identity \eqref{LJ} (see also Corollary \ref{cor:convex}).
On the other hand, from \eqref{eq:ELrs},
we immediately deduce that  both $\dot r^2(t)$ and $r^2(t)\dot s^2(t)$
tend to 0 as $t \to +\infty$. Since $r(t)$, by assumption, diverges
the second assertion follows.
\end{proof}

In the next theorem
we prove the asymptotic estimates for parabolic solutions as time diverges
(see for instance \cite{Chazy2,Chazy1,Chenciner1998,SaaHulk81,Saari71}).
The proof we propose (that can be extended to non necessarily homogeneous potentials)
is different from the classical ones, and it is similar to the one that the
first two authors exploited for the asymptotic behavior near collisions
for $N$-body type systems in \cite{BFT2}.
\begin{theorem}\label{teo:stime}
Let us suppose that $x=rs$ satisfies both
\eqref{eq:dynsys} and \eqref{eq:dynsys2} on $(t_0,+\infty)$.
Then there exists $\gamma >0$ such that:
\begin{enumerate}
\item[{\rm(a)}] $\displaystyle r(t) \sim (Kt)^{{2}/(2+\alpha)}$, as $t \to +\infty$,
                 where $K := \frac{\alpha+2}{2}\sqrt{2\gamma}$;
\item[{\rm(b)}] $\displaystyle \dot r(t) \sim \sqrt{2\gamma}(Kt)^{-\alpha/(2+\alpha)}$, as $t \to +\infty$;
\item[{\rm(c)}] $\displaystyle \lim_{t \to +\infty}V(s(t)) = \gamma$;
\item[{\rm(d)}] $\displaystyle \lim_{t \to +\infty} \nabla_T V(s(t)) = 0$;
\item[{\rm(e)}] $\displaystyle \lim_{t \to +\infty} {\rm dist}\left(C^{\gamma},s(t)\right) =0$,
                where $C^{\gamma} = \left\{ s : V(s) =\gamma, \nabla_T V(s)=0\right\}$.
\end{enumerate}
\end{theorem}
\begin{proof}
By the previous lemma, we can assume that $\dot r(t)>0$ on $(t_0,+\infty)$.\\
In order to prove {\rm(a)} we define the function
\[
\Gamma(t) := \frac12 r^{\alpha + 2}(t)|\dot s(t)|^2 - V(s(t))
           = -\frac12 r^{\alpha}(t)\dot r^2(t),
\quad t \in (t_0,+\infty)
\]
(the last equality follows from the conservation of energy).
Since $r(t)>0$ and $\dot r(t)>0$, $\Gamma(t)$ is a strictly negative and bounded quantity,
indeed
\[
-V_{\max}\leq - V(s(t)) \leq \Gamma(t) < 0.
\]
Multiplying the Euler-Lagrange equation in $s$ (see \eqref{eq:ELrs}) by $\dot s$ we obtain
\[
r^{2+\alpha} \ddot s \dot s + 2r^{1+\alpha}\dot r|\dot s|^2 - \nabla_T V(s)\dot s = 0;
\]
hence the derivative of the function $\Gamma$ satisfies, for large $t$,
\[
\dot \Gamma(t) = -\frac{2-\alpha}{2} r^{1+\alpha}(t)\dot r(t)|\dot s(t)|^2 < 0.
\]
$\Gamma (t)$ is then bounded and (strictly) decreasing; hence there exists
$\gamma >0$ such that
\[
\lim_{t \to +\infty} \Gamma (t) = -\gamma
\quad \text{ and } \quad
\lim_{t \to +\infty} r^{{\alpha}/{2}}(t)\dot r(t) = \sqrt{2\gamma}.
\]
Therefore, using de l'Hopital rule we have
\[
\lim_{t \to +\infty}\frac{r^{{\alpha}/{2}+1}(t)}{\sqrt{2\gamma}\left(\frac{\alpha}{2}+1 \right)t}
= \lim_{t \to +\infty} \frac{r^{{\alpha}/{2}}(t)\dot r(t)}{\sqrt{2\gamma}} = 1
\]
and we deduce the asymptotic behavior of $r(t)$ as $t \to +\infty$.
Straightforwardly we now prove (b), indeed, we define $K := \frac{\alpha+2}{2}\sqrt{2\gamma}$ and we obtain
\[
\lim_{t \to +\infty} \frac{\dot r(t)}{\sqrt{2\gamma} (Kt)^{-{\alpha}/(2+\alpha)} } =
\lim_{t \to +\infty} \frac{r^{{\alpha}/{2}}(t)\dot r(t)}{\sqrt{2\gamma}}
\left[\frac{(Kt)^{{2}/(2+\alpha)}}{r(t)}\right]^{{\alpha}/{2}} = 1.
\]
In order to claim (c) we remark that $\Gamma$ is bounded on $(t_0,+\infty)$, hence
its derivative has a finite integral on the same interval, that is
\[
\int_{t_0}^{+\infty}\frac{\dot r(t)}{r(t)} r^{2+\alpha}(t)|\dot s(t)|^2\dt < +\infty.
\]
Since ${\dot r(t)}/{r(t)} \sim {2}/(2+\alpha)\,t^{-1}$ as $t \to +\infty$, then necessarily
\[
\liminf_{t \to +\infty} r^{2+\alpha}(t)|\dot s(t)|^2 = 0,
\]
or, equivalently,
\[
\liminf_{t \to +\infty} V(s(t)) = \gamma.
\]
In order to conclude we need to show that also the superior limit of $V(s)$,
as $t \to +\infty$, is $\gamma$.
By the sake of contradiction let us assume that for some $C>0$
\[
\limsup_{t \to +\infty} V(s(t)) = \gamma + C,
\quad \text {that is} \quad
\limsup_{t \to +\infty} r^{2+\alpha}(t)|\dot s(t)|^2 = 2C.
\]
Then there exists a sequence $\suc{t}{k}$ such that
\[
\begin{split}
& t_k \to +\infty, \text{ as } k \to +\infty, \\
& r^{2+\alpha}(t_{2k})|\dot s(t_{2k})|^2 = \frac{2C}{3}, \forall \, k,
  \text{ and } V(s(t_{2k})) \to \frac{C}{3} + \gamma, \text{ as }k \to +\infty, \\
& r^{2+\alpha}(t_{2k+1})|\dot s(t_{2k+1})|^2 = \frac{4C}{3}, \forall \, k,
  \text{ and } V(s(t_{2k+1})) \to \frac{2C}{3} + \gamma, \text{ as } k \to +\infty, \\
& r^{2+\alpha}(t)|\dot s(t)|^2 \in \left(\frac{2C}{3},\frac{4C}{3}\right)
  \text{ for every } t \in \bigcup_k(t_{2k},t_{2k+1}).
\end{split}
\]
Using the monotone convergence of $\Gamma(t)$ to $\gamma$ we deduce that
there exists $\bar t_0 \geq t_0$ such that, for every $t \geq \bar t_0$
\[
\sqrt{\gamma} \leq r^{{\alpha}/{2}}(t)\dot r(t) \leq \sqrt{2\gamma},
\]
and, integrating on $[\bar t_0,t]$,
\begin{equation*}
\frac{2+\alpha}{2}\sqrt{\gamma}\,t + C_{\sqrt{\gamma}}
\leq r^{(2+\alpha)/2}(t)
\leq \frac{2+\alpha}{2}\sqrt{2\gamma}\,t + C_{\sqrt{2\gamma}},
\end{equation*}
where $C_{\eta}:= r^{(2+\alpha)/2}(\bar t_0)-\bar t_0\frac{2+\alpha}{2}\eta$,
$\eta \in \{\sqrt{\gamma},\sqrt{2\gamma} \}$.
We will obtain a contradiction using the properties of the sequence $\suc{t}{k}$ and the previous two estimates.
Indeed, on one hand we have
\begin{multline*}
+\infty > \int_{t_0}^{+\infty}\frac{\dot r(t)}{r(t)} r^{2+\alpha}(t)|\dot s(t)|^2\dt
        \geq \sum_k \int_{t_{2k}}^{t_{2k+1}} \frac{\dot r(t)}{r(t)} r^{2+\alpha}(t)|\dot s(t)|^2\dt \\
        \geq \frac{2C}{3} \sum_k \log \frac{r(t_{2k+1})}{r(t_{2k})}
        \geq \frac{4C}{3(2+\alpha)} \sum_k \log \frac
          { (2+\alpha)\sqrt{\gamma}\,t_{2k+1} + 2C_{\sqrt{\gamma}}}
          { (2+\alpha)\sqrt{2\gamma}\,t_{2k} + 2C_{\sqrt{2\gamma}}}.
\end{multline*}
On the other hand, since the continuity of $V(s(\cdot))$ implies
the existence of $M >0$ such that $|s(t_{2k})-s(t_{2k+1})| > M$ for every $k$,
we have
\begin{multline*}
M^2 < |s(t_{2k})-s(t_{2k+1})|^2 \leq \left(\int_{t_{2k}}^{t_{2k+1}} |\dot s(t)| \dt \right)^2\\
    \leq \int_{t_{2k}}^{t_{2k+1}} \frac{\dot r(t)}{r(t)} r^{2+\alpha}(t)|\dot s(t)|^2\dt
         \int_{t_{2k}}^{t_{2k+1}} \frac{\dt}{r^{1+\alpha}(t)\dot r(t)}.
\end{multline*}
Since
\[
r^{1+\alpha}(t)\dot r(t) = r^{{\alpha}/{2}}(t)\dot r(t)\cdot r^{(2+\alpha)/2}(t)
\geq  \frac{\sqrt{\gamma}}{2}\left((2+\alpha) \sqrt{\gamma}\,t + 2C_{\sqrt{\gamma}}\right),
\]
then
\begin{multline*}
\int_{t_{2k}}^{t_{2k+1}} \frac{\dt}{r^{1+\alpha}(t)\dot r(t)} \leq
\frac{2}{\sqrt{\gamma}}\int_{t_{2k}}^{t_{2k+1}} \frac{\dt}{(2+\alpha)\sqrt{\gamma}\,t + 2C_{\sqrt{\gamma}}}\\
= \frac{2}{(2+\alpha)\gamma}
\log \frac{(2+\alpha) \sqrt{\gamma}\,t_{2k+1} + 2 C_{\sqrt{\gamma}}}
          {(2+\alpha) \sqrt{\gamma}\,t_{2k} + 2 C_{\sqrt{\gamma}}}
\end{multline*}
and
\begin{multline*}
+\infty  > \int_{t_0}^{+\infty}\frac{\dot r(t)}{r(t)} r^{2+\alpha}(t)|\dot s(t)|^2\dt 	\\
         \geq M^2 \frac{(2+\alpha)\gamma}{2} \sum_k \left[
        \log \frac{(2+\alpha) \sqrt{\gamma}\,t_{2k+1} + 2C_{\sqrt{\gamma}}}
          {(2+\alpha) \sqrt{\gamma}\,t_{2k} + 2C_{\sqrt{\gamma}}} \right]^{-1}
\end{multline*}
Hence,
\begin{multline*}
+\infty > \int_{t_0}^{+\infty}\frac{\dot r(t)}{r(t)} r^{2+\alpha}(t)|\dot s(t)|^2\dt
 \geq \frac{2C}{3(2+\alpha)} \sum_k \log \frac
          { (2+\alpha)\sqrt{2\gamma}\,t_{2k+1} + 2C_{\sqrt{2\gamma}}}
          { (2+\alpha)\sqrt{\gamma}\,t_{2k} + 2C_{\sqrt{\gamma}}}      \\
 + M^2 \frac{(2+\alpha)\gamma}{4} \sum_k \left[
        \log \frac{(2+\alpha) \sqrt{\gamma}\,t_{2k+1} + 2C_{\sqrt{\gamma}}}
          {(2+\alpha) \sqrt{\gamma}\,t_{2k} + 2C_{\sqrt{\gamma}}} \right]^{-1}.
\end{multline*}
We obtain a contradiction when we impose the convergence of both series
(with positive terms) at the right hand side.

We now turn to assertion (d), that is equivalent to
\[
\lim_{t \to +\infty} r^{2+\alpha}(t) \ddot s(t)=0,
\]
indeed in the Euler-Lagrange equation in the variable $s$ (see \eqref{eq:ELrs}) both terms
$-2r^{1+\alpha}\dot r\dot s = -2r^{(2+\alpha)/2}\dot s\,r^{{\alpha}/{2}}\dot r$
and $-r^{2+\alpha}|\dot s|^2 s$ are infinitesimal as $t \to +\infty$
(indeed $r^{2+\alpha}|\dot s|^2$ is infinitesimal and $r^{{\alpha}/{2}}\dot r$ remains bounded).
We proceed by contradiction: suppose there exists
$\suc{t}{k}$ such that, as $k \to +\infty$, $t_k \to +\infty$ and $\nabla_TV(s(t_k)) \not \to 0$.
Up to subsequences
we deduce the existence of $\bar s \in \sphere$ and $\sigma \neq 0$ such that
\[
s(t_k) \to \bar s,\text{ as $k \to +\infty$ and } \nabla_TV(\bar s) = \sigma.
\]
For any $h>0$ and $\ve>0$ we have that, since, for $t$ large, $|\dot s(t)|<\ve$,
\[
|s(t)-s(t_k)|<\ve h, \quad \forall t \in \bigcup_k [t_k,t_k+h]
\]
and
\[
|s(t)-\bar s| \leq |s(t)-s(t_k)|+|s(t_k)-\bar s|<\ve (h+1), \quad \forall t \in \bigcup_k [t_k,t_k+h].
\]
The convergence of $s(t)$ to $\bar s$ is then uniform on $\bigcup_k [t_k,t_k+h]$,
for any $h>0$, hence, by continuity,
\[
\sup_{t \in [t_k,t_k+h]} |\nabla_TV(s(t))-\sigma|  \to 0 \quad \text{and} \quad
\sup_{t \in [t_k,t_k+h]} |r^{2+\alpha}(t) \ddot s(t)-\sigma| \to 0,
\]
as $k \to +\infty$.
In order to obtain a contradiction we perform the time scaling (see \cite{McG1974})
\[
\tau = \int_{t_0}^{+\infty} \frac{\dt}{r^{(2+\alpha)/2}(t)}
\]
which maps $[t_0,+\infty)$ into $[0,+\infty)$ (we have used here the asymptotic for $r(t)$).
Letting $^\prime{}:= {\d}/{\d\tau}$
we have
\[
\lim_{\tau \to +\infty} |s'(\tau)|^2 = \lim_{t \to +\infty} r^{2+\alpha}(t)|\dot s(\tau)|^2 = 0
\]
while, by contradiction, for any fixed $h>0$
\[
\sup_{t \in [t_k,t_k+h]} |r^{2+\alpha}(t) \ddot s(t)-\sigma| =
\sup_{\tau \in [\tau_k,\tau_k+\tilde h]} |s''(\tau)-\sigma|  \to 0, \quad \text{as $k \to +\infty$}.
\]
We obtain the contradiction
\[
0 = \lim_{k \to +\infty} |s'(\tau_k +h) - s'(\tau_k)|
  = \lim_{k \to +\infty} \left| \int_{\tau_k}^{\tau_k+h}s''(\tau) \dtau \right|
  = h |\sigma| \neq 0.
\]
Finally, assertion (e) follows directly from (c) and (d).
\end{proof}
\begin{corollary}\label{coro:lim_r}
Let us suppose that $x=rs$ satisfies both
\eqref{eq:dynsys} and \eqref{eq:dynsys2} on $(-\infty,t_*)\cup(t_{**},+\infty)$.
Then there exist constants $\Upsilon_\pm>0$ such that
\[
\lim_{t \to \pm\infty} \frac{r^{(2+\alpha)/2}(t)}{t} = \Upsilon_\pm>0
\]
\end{corollary}
\end{appendices}
%

%\nocite{BFT2,Chazy1,Chazy2,Chenciner1998,MadVen2009,MarSaari76,PollardBook,Saari71,SaaHulk81}
%\bibliography{vivinabibliog}
%\bibliographystyle{siam}

\end{document}